\documentclass[11pt,a4paper,twoside]{article}
\usepackage{amscd}
\usepackage{enumerate}

\usepackage{amsmath, amssymb, latexsym}
\usepackage[all]{xy}
\usepackage{bbm}
\usepackage{authblk}
\usepackage{amsfonts}
\usepackage{amsthm}
\usepackage{relsize}
\usepackage{setspace}
\usepackage{geometry}
\usepackage{url}
\usepackage{xspace}
\usepackage{tocloft}
\usepackage{graphics}
\usepackage{graphicx}
\usepackage{lscape}
\usepackage{microtype}
\usepackage[T1]{fontenc}
\usepackage{ulem}
\usepackage{mathrsfs}

\usepackage[usenames, dvipsnames]{color}
\usepackage[utf8]{inputenc}
\usepackage{tikz}

\usepackage[pagebackref=true]{hyperref}
\usepackage[alphabetic]{amsrefs}
\usepackage[english]{babel}

\newtheorem{theorem}{Theorem}[section]
\newtheorem{proposition}[theorem]{Proposition}

\newtheorem{corollary}[theorem]{Corollary}
\newtheorem{lemma}[theorem]{Lemma}
\newtheorem{fact}[theorem]{Fact}

\theoremstyle{definition}
\newtheorem{definition}[theorem]{Definition}

\newtheorem{remark}[theorem]{Remark}
\newtheorem{remarks}[theorem]{Remarks}

\newtheorem*{rema*}{Remark}
\newtheorem*{remas*}{Remarks}
\newtheorem*{prop*}{Proposition}
\newtheorem*{theorem*}{Theorem}

\theoremstyle{problem}

\newcommand{\Aut}{\mathrm{Aut}}

\newcommand{\Ker}{\mathrm{Ker}}

\newcommand{\RR}{\mathbf{R}}

\newcommand{\CAT}{\mathrm{CAT}}
\newcommand{\cat}{$\mathrm{CAT}(0)$\xspace}

\newcommand{\Ch}{\mathrm{Ch}}

\newcommand{\Stab}{\mathrm{Stab}}

\newcommand{\bd}{\partial}

\newcommand{\Id}{\operatorname{Id}}

\newcommand{\Top}{\operatorname{\mathrm{Top}}}

\newcommand{\A}{\mathrm{\bold{A}}}

\newcommand{\efface}[1]{}

\def\og{\leavevmode\raise.3ex\hbox{$\scriptscriptstyle\langle\!\langle$~}}
\def\fg{\leavevmode\raise.3ex\hbox{~$\!\scriptscriptstyle\,\rangle\!\rangle$}}

\def\ie{{\it i.e.,\/}\ }
\def\eg{{\it e.g.,\/}\ }
\def\resp{{\it resp.,\/}\ }
\def\etc{{\it etc.\/}}
\def\lc{{\it l.c.\/}\ }

\def\truc{\unskip\kern 3pt\penalty 500
\hbox{\vrule\vbox to 5pt{\hrule width 4pt\vfill\hrule}\vrule}\kern
3pt}

\def\into{\hookrightarrow}

\def\vect{\overrightarrow}

\def\parni{\par\noindent}

\def\Z{{\mathbb Z}}

\def\R{{\mathbb R}}

\def\A{{\mathbb A}}
\def\M{{\mathbb M}}

\newcommand{\g}[1]{\mathfrak{#1}} 

\def\qa{\alpha}     
\def\qb{\beta}
\def\qd{\delta}

\def\qf{\varphi}
\def\qg{\gamma}
\def\qi{\iota}
 
 \def\ql{\lambda}
\def\qm{\mu}

\def\qp{\pi}
\def\qr{\rho}
\def\qs {\sigma}

\def\qx{\xi}

\def\QD{\Delta}
\def\QF{\Phi}

\def\QL{\Lambda}
\def\QO{\Omega}

\def\QS{\Sigma}

\def\sha{{\mathcal A}}   

\def\shd{{\mathcal D}}

\def\shf{{\mathcal F}}

\def\shk{{\mathcal K}}

\def\shm{{\mathcal M}}

\def\sht{{\mathcal T}}

\title{The cone topology on masures}

\author[1]{Corina Ciobotaru\thanks{corina.ciobotaru@unifr.ch, Partially supported by Swiss National Science Foundation Grant 153599}}
\author[2]{Bernhard M\"{u}hlherr\thanks{Bernhard.M.Muehlherr@math.uni-giessen.de}}
\author[3,4]{Guy Rousseau\thanks{Guy.Rousseau@univ-lorraine.fr}}
\author[5]{\\ with an appendix by Auguste H\'ebert\thanks{auguste.hebert@univ-st-etienne.fr, Partially supported by ANR grant ANR-15-CE40-0012}}

\affil[1]{Universit\'e de Fribourg, Section de Math\'ematiques, Chemin du Mus\'{e}e 23, 1700 Fribourg, Switzerland.}
\affil[2]{Universit\"{a}t Giessen, Mathematisches Institut, Arndtstrasse 2, 35392 Gie{\ss}en, Germany.}
\affil[3]{Universit\'e de Lorraine, Institut \'Elie Cartan de Lorraine, UMR 7502, Vand\oe uvre-l\`es-Nancy, F-54506, France.}
\affil[4]{CNRS, Institut \'Elie Cartan de Lorraine, UMR 7502, Vand\oe uvre-l\`es-Nancy, F-54506, France.}
\affil[5]{Universit\'e Lyon, UJM-Saint-\'Etienne CNRS, UMR 5208 CNRS, Saint-\'Etienne, F-42023,  France}

\date{June 11, 2018}

\begin{document}

\newcounter{qcounter}

\maketitle

\begin{abstract}

Masures are generalizations of Bruhat--Tits buildings and the main examples are associated with almost split Kac--Moody groups $G$ over non-Archimedean local fields. In this case, $G$ acts strongly transitively on its corresponding masure $\Delta$ as well as on the building at infinity of $\Delta$, which is the twin building associated with $G$.
The aim of this article is twofold: firstly, to introduce and study the cone topology on the twin building at infinity of a masure. It turns out that this topology has various favorable properties that are required in the literature as axioms for a topological twin building. Secondly, by making use of the cone topology, we study strongly transitive actions of a group $G$ on a masure $\Delta$.
Under some hypotheses, with respect to the masure and the group action of $G$, we prove that $G$ acts strongly transitively on $\Delta$ if and only if it acts strongly transitively on the twin building at infinity $\partial \Delta$. Along the way a criterion for strong transitivity is given and the existence and good dynamical properties of strongly regular hyperbolic automorphisms of the masure are proven.
\end{abstract}

\section{Introduction}


Masures were introduced by Gaussent and Rousseau~\cite{GR08} and developed further by Rousseau~\cite{Rou11, Rou12, Rou13}, Charignon~\cite{Cha}, Gaussent--Rousseau~\cite{GR14} and Bardy-Panse--Gaussent--Rousseau in~\cite{BPGR14}. The theory of masures has been developed in order to make several techniques from  Bruhat--Tits theory also applicable in the context of Kac--Moody groups.
By Bruhat--Tits theory, with every semi-simple algebraic group $G$ over a non-Archimedean local field $\shk$ there is associated  a ``symmetric space'', called the affine building of $G$. More generally, with an almost split Kac--Moody group $G$ over such a field $\shk$ there is associated a masure whose construction is carried out by suitably modifying the Bruhat--Tits construction.  As for buildings, a masure is covered by apartments corresponding to the maximal split tori and every apartment is a finite dimensional real affine space endowed with a set of affine hyperplanes called walls. At a first glance, a masure looks very much like a  Bruhat--Tits building. However, in the case of masures, the set of all walls is not always a locally finite system of hyperplanes and  it is no longer true that any two points are contained in a common apartment. For this reason, the word ``building'' is replaced by the word ``masure''.  Moreover, if we look at infinity of a masure, we obtain a twin building, which is the analogue of the spherical building at infinity associated with a Bruhat--Tits building.
As an application, Bardy-Panse, Gaussent and Rousseau~\cite{GR14},~\cite{BPGR14}  use the construction of masures to define the spherical Hecke and Iwahori--Hecke algebras associated with an almost split Kac--Moody group over a non-Archimedean local field.
In~\cite{GR08} the masure is used to make the link (in the representation theory of Kac--Moody groups) between the  Littelmann's path model and the Mirkovic--Vilonen cycle model. 
{As described above, the theory of masures is a natural generalization
of Bruhat--Tits theory to the Kac--Moody situation. As Bruhat--Tits buildings can be interpreted as symmetric spaces over non-Archimedean
local fields it is appropriate to mention here as well that a theory of symmetric spaces for Kac--Moody groups over Archimedean local fields has been recently developed by Freyn, Hartnick, Horn and K{\"o}hl in \cite{FHHK}. }


Although there {are} several contributions to the subject, it is fair to say that
the theory of masures is not yet in its final form. This is underlined by the fact
that there are several basic questions that still need to be answered. Most of these questions
are about generalizations of known results from Bruhat--Tits theory and sometimes it is even
not clear how to state properly the analogue of such a result in the context of masures.

In~\cite{CR15} two of us provided a generalization  of a result
on {strongly transitive actions} that has been proved by the first author in a joint work with P.-E. Caprace in~\cite{CaCi} for Bruhat--Tits buildings. The result in~\cite{CaCi} involves the cone topology on the spherical building at infinity of a Bruhat--Tits building and therefore relies on the fact that the latter is a \cat-space. Unfortunately, it seems that there is no reasonable way to endow a masure with a \cat-metric.
However, it turns out
that there is nonetheless a canonical cone topology on the
set of chambers of the twin building at infinity of a masure. This observation was indeed
one of the crucial insights in~\cite{CR15} to generalize the result in~\cite{CaCi} to masures.

Apart from being useful in this particular context
we believe that the cone topology is an important ingredient for the theory of masures.
For this reason we provide several basic results regarding the cone topology
which show that it has several remarkable properties. As a consequence of our
{investigation}, it turns out that the twin building at infinity of a masure endowed with the cone topology is a ``weak topological twin building'' in the sense of Hartnick--K\"{o}hl--Mars~\cite{HKM}. In that paper, topological twin buildings
are constructed by means of the Kac--Peterson topology of a Kac--Moody group.
It turns out that masures provide examples of weak topological
twin buildings {and this} construction appears to be much more natural and direct than the one given in Hartnick--K\"{o}hl--Mars~\cite{HKM}.

As a result of what has been said so far, the goal of this article is twofold.
Firstly, we introduce the cone topology on the set of chambers of the twin
building at infinity of a masure and study its basic properties. Secondly, we investigate
the notion of strong transitivity for masures. Based on the notion of the
cone topology we are able to provide a generalization of a result of Caprace and Ciobotaru from~\cite{CaCi}.

\subsection{The cone topology on a masure}

\medskip
For any \cat space (\eg a locally finite affine building, a locally finite tree, \etc) the cone topology is a natural topology used  to study various properties of the visual boundary of the corresponding \cat space. Unfortunately, there does not seem to be any reasonable way to endow a masure with a \cat metric.
Therefore it is unlikely that we can use the full power of the \cat machinery in the  context
of masures. However, it turns out that there is still a canonical topology on the set of
chambers of the twin building at infinity of a masure which generalizes the cone topology
on the visual boundary of a Bruhat--Tits building. The precise definition of the cone topology
requires some preparation and it is given in Definition~\ref{def::cone_top_on_chambers}.
 It turns out that the cone topology has several remarkable properties and some of them are given in the following theorem.

\begin{theorem}
\label{theo:cone.top}
Let $\Delta$ be a masure and let $X$ be the set of chambers of the twin building at infinity of $\Delta$.
\begin{enumerate}
\item The cone topology on the set $X$ of all chambers at infinity of a masure is  Hausdorff.
\item
Residues are closed with respect to the cone topology on $X$.

\item
The set of chambers in $X$ opposite a given chamber in $X$ is open in $X$.

\item
When the masure is ``locally finite'' and ``locally complete'' (see Definitions~\ref{def::locally_finite} and~\ref{def::locally_complete} below), then any panel and, more generally, any spherical residue or Schubert variety in $X$ is compact.

\end{enumerate}
\end{theorem}


\par We would like to point out that the statement of Theorem \ref{theo:cone.top} is meant
to give the reader a rough idea about our results on the cone topology. They will be proved
among other results given in Sections~\ref{subsec::prop_cone_top} and~\ref{subsec::further_prop} to which we refer for the details. As already indicated above,
these results provide a link to previous contributions of Kramer~\cite{K02} and
Hartnick--K\"ohl--Mars~\cite{HKM} on topological twin buildings. Especially in the latter
reference the authors discuss several axioms that one could require for a topological twin building.
It turns out that the cone topology behaves very nicely in the sense that it satisfies most
of them.
Our results about the basic properties of the  cone topology have been recently complemented
by a contribution of Auguste H\'ebert that is included in this paper as an appendix.

\par Apart from the basic facts on the cone topology recorded in Theorem~\ref{theo:cone.top} we shall also present in this introduction a slightly more specific result. This result relies on the notion of a panel tree in a masure, which also plays a crucial role on our results on strongly transitive actions. The concept of a panel tree is well known from the theory of affine buildings and it turns out that it generalizes without any problems to masures. More precisely, let $\Delta$ be a masure and let $X$ be its  building at infinity. Then one can associate with any pair $(R,R')$ of opposite spherical residues of $X$ an affine building $\QD(R,R')$ which is a ``convex'' subset of $\Delta$; moreover, there is a canonical identification of $R$ with the  spherical building at infinity of $\QD(R,R')$. As we already mentioned, there is the ``usual'' cone topology on the boundary $\partial \QD(R,R')$ of $\QD(R,R')$ which induces a topology $\tau_{R'}$ on $R$, that a priori depends on the opposite residue $R'$.

\begin{proposition}
\label{prop:cone_top_residue}
Let $\Delta$ be a masure, let $X$ be its  building at infinity, let $R$ be a spherical residue of $X$
and let $\tau_X$ be the cone topology on $X$. Then $\tau_X \vert_R = \tau_{R'}$ for each residue $R'$ that is opposite
 $R$ in $X$.
\end{proposition}

For more details about Proposition~\ref{prop:cone_top_residue} see Section~\ref{1.5} and Remark~\ref{rem::essential}.

\subsection{A topological criterion for strong transitivity}
\label{sec::1.2}

In the second part of the paper we shall provide a criterion which ensures that a group
action on a masure is strongly transitive. This criterion involves the cone topology on
the twin building at infinity. Its proof relies on a modification of the strategy followed
by Caprace and Ciobotaru in \cite{CaCi}.

\par The notion of a strongly transitive action on a discrete building is crucial for group theoretic applications of buildings since it provides the link to groups with BN-pairs.
 If the building is no longer discrete (\eg $\R-$buildings), the standard definition of a strongly transitive action has to be suitably modified in order to adapt the basic machinery (see Section \ref{def:1.2} below).
 Thus it is not surprising that there is also a natural definition of a strongly transitive action on a masure. It is given in Rousseau \cite[4.10]{Rou13} (see also Gaussent-Rousseau \cite[1.5]{GR14}) and it is recalled in Definition \ref{4.2} below. As this definition is somehow involved we prefer to omit it in this introduction.
The starting point of our investigation is provided by the following:

\begin{fact}(see Proposition \ref{5.5})
\label{fact}
Let $\QD$ be a thick affine building (or a thick masure) and let $X$ be its (twin)-building at infinity.
Suppose that a group $G$ acts strongly transitively on $\QD$, then the action of $G$ induced on $X$ is strongly transitive as well.
 \end{fact}

\par The converse is not true in general: There are examples of group actions on affine buildings or masures which are not strongly transitive on $\QD$ but which induce a strongly transitive action on $X$.
The following remark provides a basic example of such a situation.

\begin{remark}\label{rem::DTE_not_ST}  Let $T$ be a regular tree and let $\qx$ be an end of $T$.
Define $\sha$ to be the set of all apartments of $T$ not containing $\qx$ at infinity and let $G$ be the group of automorphisms of $T$ fixing $\qx$.
 Then $G$ is $2-$transitive (equivalently, strongly transitive) on the set $E(\sha)$
of ends of $(T,\sha)$.
  But $G$ is not strongly transitive on  $(T,\sha)$: any $g \in G$ stabilizing an
apartment $A\in \sha$ fixes the projection of $\qx$ on $A$.
  \end{remark}

Suppose that $G$ is a group acting on a masure such that its action on the twin building at infinity is strongly transitive.
  In view of the example above it is natural to ask for additional conditions which ensure that $G$ acts strongly transitively on the masure.
  For affine buildings one such condition is the assumption that the apartment system of $\QD$ is complete.
  In the locally finite case this is due to Caprace and Ciobotaru, as it follows from
Theorem 1.1 in \cite{CaCi}.
   It was proved later in full generality by Kramer and Schillewaert
in \cite{KS}.

 \par Although completeness of the apartment system is a most natural assumption, there are important situations where the apartment system is not complete.
 Moreover, in the case of masures it is even not clear whether there is a sensible definition of a complete apartment system.
 Thus, it is natural to ask for other conditions which ensure that a strongly transitive action at infinity yields a strongly transitive action on the affine building or masure.
 In this paper we provide such a condition for masures.
 It involves the cone topology and the notion of a strongly regular hyperbolic element in the automorphisms group of a masure.
  It relies on the observation, that the strategy used by Caprace and Ciobotaru in
the proof of Theorem 1.1 in \cite{CaCi} can be modified in such a way that completeness is not needed and - more importantly in our context - that it also works for masures.
  Before stating our generalization of the main result of \cite{CaCi} it is convenient to provide some information about strongly regular hyperbolic elements in the automorphism  group of a masure and to formulate the analogue of the topological condition in \cite{CaCi} for masures.

  \medskip
\parni{\bf Strongly regular hyperbolic automorphisms of masures}:
   Let $\QD$ be a masure and let $\qa\in \Aut(\QD)$.
    Then $\qa$ is called a strongly regular hyperbolic automorphism of $\QD$ if it has no fixed points in $\QD$ and stabilizes an apartment $A$ of $\QD$ on which it induces a translation.
Moreover, the direction of the latter is required to be in the Tits cone (up to sign), but should not be contained in any of its walls (see Definition \ref{def::str_reg_lines}).

\par  As a side product of our investigation in this paper we obtain the following general fact about strongly regular hyperbolic elements:

\begin{proposition}\label{STimpliesreghyp}  Let $\QD$ be a {thick} masure such that the type of $\bd\QD$ has no direct factor of affine type.
If a {vectorially Weyl subgroup} $G\leq\Aut(\QD)$ is strongly transitive, then $G$ contains a strongly regular hyperbolic element.
\end{proposition}

\par This is a consequence of Lemma \ref{5.7} and Theorem \ref{thm:ExistenceStronglyReg} below.

\medskip
\parni{\bf A topological condition for group actions on masures}: The following definition is motivated by one of the assumptions in Theorem 1.1 of \cite{CaCi}:

\begin{definition}\label{1.6}
Let $\QD$ be a masure and let $G$ be a group acting  by vectorially
Weyl automorphisms {(see Section 2.4.1)} on $\QD$.
 Let $c\in \Ch(\bd\QD)$, let $c^{op}$ denote the set of chambers in $\bd\QD$ which are opposite to $c$ and let $G_c$ denote the stabilizer of $c$ in $G$.
  The \textbf{horospherical stabilizer} of $c$ is defined to be  the group
  $G_c^0:=\{ g\in G_c \mid g \mathrm{\ fixes\ a\ point\ in\ \QD}\}$ (see Lemma \ref{lem:unip}).

  \medskip
\par We say that $G$ satisfies Condition (Top) if, for each chamber $c$ in $\bd\QD$, the following holds:

(Top) The orbits of $G_c^0$ in $c^{op}$ are closed {with respect to} the topology on $c^{op}$ that is induced from the cone topology on $\bd\QD$.
\end{definition}

\par We are now in the position to state our main result  about strongly transitive actions on masures:

\begin{theorem}
\label{thm::main_theorem_intr}
Let $(\Delta, \sha)$ be a thick masure and let $G$ be a vectorially Weyl subgroup of $\Aut(\Delta)$.
If $G$ contains a strongly regular hyperbolic element and satisfies Condition (Top), then the following are equivalent:

\begin{enumerate}[(i)]
\item
\label{thm:main-thm-i_int}
$G$ acts strongly transitively on $\Delta$;

\item
\label{thm:main-thm-ii_int}
$G$ acts strongly transitively on the twin building at infinity $\partial \Delta$.
\end{enumerate}
\end{theorem}

Combining Theorem \ref{thm::main_theorem_intr}, Proposition \ref{STimpliesreghyp} and Remark \ref{rem::semi-discretness}, we obtain also the following:

\begin{corollary}\label{cornoaffinedirectfactors}  Let $\QD$ be a thick masure such that the type of $\bd\QD$ has no direct factor of affine type.
Let $G\leq \Aut(\QD)$ be a vectorially Weyl subgroup.
Then the following are equivalent:

(i) $G$ acts strongly transitively on $\bd\QD$, contains a strongly regular hyperbolic automorphism of $\Aut(\QD)$ and satisfies Condition (Top);

(ii) G is strongly transitive on $\QD$.
\end{corollary}

\begin{remark}\label{1.9} Theorem \ref{thm::main_theorem_intr} generalizes Theorem 1.1 of \cite{CaCi}. But in the latter theorem, there is a third equivalent condition involving the existence of a commutative spherical Hecke algebra.
 This existence is proved in \cite{GR14} for groups strongly transitive on a locally finite masure. We did not investigate here the possibility of a converse result.
\end{remark}

\subsection{Organization of the paper}
\label{sec::1.3}

\par In Section \ref{s1} we explain the background of this article  about masures and twin buildings.
We add also a less known topic: the affine building (\resp panel tree) associated to a pair of opposite ideal faces (\resp panels) at infinity of a masure.

In Section \ref{sec::cone_top_hovel} we define and study the cone topology on the twin building at infinity of a masure.
 We prove Theorem \ref{theo:cone.top} and several other results
 {which are motivated by the axioms for topological twin buildings discussed in} 
  \cite{HKM}.

 In Section \ref{s2} we introduce different notions of stabilizers (in particular the horospherical stabilizer) of chambers or pairs of opposite panels at infinity, and prove some technical lemmas about them.
 We remind also the existing definitions of strong transitivity in combinatorial buildings, affine $\R-$buildings and twin buildings.

 The definition of strong transitivity in masures is introduced in Section \ref{s4}.
 Two equivalent simpler definitions are given and the Fact \ref{fact} is proven.
 A link is also made with a weaker condition (LST) of ``{locally} strong transitivity''.

 In Section \ref{sec::str_reg_elements} we study the existence and dynamics of strongly regular hyperbolic elements.
 We prove Proposition \ref{STimpliesreghyp}, under the weaker hypothesis (LST) and more generally for masures of affine type, under a technical condition (AGT).
 This allows us to prove Theorem \ref{thm::main_theorem_intr}  and Corollary \ref{cornoaffinedirectfactors}  in Section \ref{sec::main_theorem}.

\bigskip
\subsection*{Acknowledgement}
This paper is an extension of the preprint \cite{CR15} which was initiated when G.R. was visiting C.C. at the University of Geneva in 2014. C.C. and G.R. would like to thank this institution for its hospitality and good conditions of working. The research yielding to the paper in its present form was accomplished during a visit of C.C. and G.R. at the University of Gie{\ss}en. The authors are grateful to this institution for its support.

\section{Masures}
\label{s1}

\subsection{Vectorial data}
\label{1.1}  We consider a quadruple $(V,W^v,(\qa_i)_{i\in I}, (\qa^\vee_i)_{i\in I})$ where $V$ is a finite dimensional real vector space, $W^v$ a subgroup of $GL(V)$ (the vectorial Weyl group), $I$ a finite non empty set, $(\qa^\vee_i)_{i\in I}$ a family in $V$ and $(\qa_i)_{i\in I}$ a free family in the dual $V^*$.
 We ask these data to satisfy the conditions of Rousseau~\cite[1.1]{Rou11}.
  In particular, the formula $r_i(v)=v-\qa_i(v)\qa_i^\vee$ defines a linear involution in $V$ which is an element in $W^v$ and $(W^v,S=\{r_i\mid i\in I\})$ is a Coxeter system.

  \par Actually we consider throughout the paper the Kac--Moody case of \cite[1.2]{Rou11}: the matrix $\M=(\qa_j(\qa_i^\vee))_{i,j\in I}$ is a generalized (possibly decomposable) Cartan matrix.
  Then $W^v$ is the Weyl group of the corresponding Kac--Moody Lie algebra $\g g_\M$, it acts on $V$ and its dual $V^*$ with the associated real root system being
$$
\QF=\{w(\qa_i)\mid w\in W^v,i\in I\}\subset Q=\bigoplus_{i\in I}\,\Z.\qa_i.
$$
We set $\QF^\pm{}=\QF\cap Q^\pm{}$ where $Q^\pm{}=\pm{}(\bigoplus_{i\in I}\,(\Z_{\geq 0}).\qa_i)$ and $Q^\vee=(\sum_{i\in I}\,\Z.\qa_i^\vee)$, $Q^\vee_\pm{}=\pm{}(\sum_{i\in I}\,(\Z_{\geq 0}).\qa_i^\vee)$. We have  $\QF=\QF^+\cup\QF^-$ and, for $\qa=w(\qa_i)\in\QF$, $r_\qa=w.r_i.w^{-1}$ and $r_\qa(v)=v-\qa(v)\qa^\vee$, where the coroot $\qa^\vee=w(\qa_i^\vee)$ depends only on $\qa$.  We shall also consider the imaginary roots of $\g g_\M$: $\QF_{im}=\QF_{im}^+\cup\QF_{im}^-$ where $\QF_{im}^\pm\subset Q^\pm$ 
   is $W^v-$stable and $\QF\sqcup\QF_{im}$ is the set of all roots of  $\g g_\M$, associated with some Cartan subalgebra.


  \par The {\bf fundamental positive chamber} is $C^v_f:=\{v\in V\mid\qa_i(v)>0,\forall i\in I\}$.
   Its closure $\overline{C^v_f}$ is the disjoint union of the {\bf vectorial faces} $F^v(J):=\{v\in V\mid\qa_i(v)=0,\forall i\in J,\qa_i(v)>0,\forall i\in I\setminus J\}$ for $J\subset I$.
    The {\bf positive (\resp negative)} vectorial faces are the sets $w.F^v(J)$ (\resp $-w.F^v(J)$) for $w\in W^v$ and $J\subset I$; they are {\bf chambers} (\resp {\bf panels}) when $J=\emptyset$ (\resp $\vert J\vert=1$).
    The {\bf support} of such a face is the vector space it generates.
    The set $J$ or the face $w.F^v(J)$ or an element of this face is called {\bf spherical} if the group  $W^v(J):=\langle r_i\mid i\in J \rangle$ is finite.
    Actually the {\bf type} of the face $\pm w.F^v(J)$ is $J$; by abuse of notation we identify $J$ with the subset $\{ r_i\mid i\in J \}$ of $S$.
    A chamber or a panel is spherical.

\par The {\bf Tits cone}  $\sht$ (\resp its {\bf interior}  $\sht^\circ$) is the (disjoint) union of the positive vectorial (\resp  and spherical) faces.
{Actually $\cal T$ is a geometric realization of the Coxeter complex of $W^v$. The classical geometric realization is the quotient by $V_0=F^v(I)=\cap_{i\in I} Ker(\alpha_i)$ and $\R_{>0}$ of ${\cal T}\setminus V_0$.}

\par We make no irreducibility hypothesis for $(V,W^v)$. So $V$ (and also $\A$, $\QD$ as below) may be a product of direct irreducible factors, which are either of finite, affine or indefinite type, see Kac~\cite[4.3]{K90}.

\par When $\M$ is a {genuine} Cartan matrix (\ie when $W^v$ is finite) $\g g_\M$ is a reductive Lie algebra and everything is classical, in particular $\QF$ is finite, $\QF_{im}=\emptyset$ and $\sht=\sht^\circ=V$.

\subsection{The model apartment}\label{1.2} As in Rousseau~\cite[1.4]{Rou11} the model apartment $\A$ is $V$ considered as an affine space and endowed with a family $\shm$ of walls. 
 These walls  are affine hyperplanes directed by Ker$(\qa)$, for $\qa\in\QF$.
 They can be described as $M(\qa,k)=\{v\in V\mid\qa(v)+k=0\}$, for $\qa\in\QF$ and $k\in\QL_\qa\subset\R$ (with $\vert\QL_\qa\vert=\infty$). 
 We may (and shall) suppose the origin to be \textbf{special}, \ie $0\in \QL_\qa$, $\forall\qa\in\QF$.
 We define $\QL_\qa=\R$, when $\qa\in\QF_{im}$.

 \par We say that this apartment is {\bf semi-discrete} if each $\QL_\qa$ (for $\qa\in\QF$) is discrete in $\R$; in this case 
 $\QL_\qa=k_\qa.\Z$ is a non trivial discrete subgroup of $\R$.
Using   Lemma 1.3 in Gaussent--Rousseau~\cite{GR14} (\ie replacing $\QF$ by another system $\QF_1$) we may assume that $\QL_\qa=\Z, \forall\qa\in\QF$.
 Notice that, for $W^v$ finite, semi-discrete is equivalent to discrete (which means that the family $\shm$ of walls is locally finite).

\par For $\qa\in\QF$, $k\in\QL_\qa$ and $M=M(\qa,k)$, the reflection $r_{\qa,k}=r_M$ with respect to the wall $M$ is the affine involution of $\A$ with fixed points this wall $M$ and associated linear involution $r_\qa$. The affine Weyl group $W^a$ is the group generated by the reflections $r_M$ for $M\in \shm$; we assume that $W^a$ stabilizes $\shm$. Actually  $W^a=W^v\ltimes Q^\vee$ when $\QL_\qa=\Z$, $\forall\qa\in\QF$; the group $W^a$ is then a Coxeter group when $W^v$ is finite.

\par An {\bf automorphism of $\A$} is an affine bijection $\qf:\A\to\A$ stabilizing the set of pairs $(M,\qa^\vee)$ of a wall $M$ and the coroot associated with $\qa\in\QF$ such that $M=M(\qa,k)$, $k\in\Z$. The group $\Aut(\A)$ of these automorphisms contains $W^a$ and normalizes it.


   \par For $\qa\in\QF$ and $k\in\R$, $D(\qa,k)=\{v\in V\mid\qa(v)+k\geq 0\}$ is an half-space, it is called an {\bf half-apartment} if $M(\qa,k)$ is a wall \ie $k\in\QL_\qa$. 


The Tits cone $\mathcal T$ and its interior $\mathcal T^o$ are convex and $W^v-$stable cones (this is proved in \cite[3.12]{K90} for $\mathcal T$ and it is then clear for $\mathcal T^o$).
 Therefore, we can define two $W^v-$invariant preorder relations  on $\mathbb A$:
$$
x\leq y\;\Leftrightarrow\; y-x\in\mathcal T
; \quad x\stackrel{o}{<} y\;\Leftrightarrow\; y-x\in\mathcal T^o.
$$
 If $W^v$ has no  fixed point in $V\setminus\{0\}$ and no finite factor, then they are order relations; still, they are not in general: one may have $x\leq y,y\leq x$ and $x\neq y$.
 If $W^v$ is finite, both relations are trivial.


\subsection{Faces, sectors, chimneys...}
\label{suse:Faces}

 The faces in $\mathbb A$ are associated with the above systems of walls
and half-apartments.  As in Bruhat--Tits~\cite{BrT72}, they
are no longer subsets of $\mathbb A$, but filters of subsets of $\mathbb A$. For the definition of that notion and its properties, we refer to Bruhat--Tits~\cite{BrT72} or Gaussent--Rousseau~\cite[Definition~2.3]{GR08}. We endow $\A$ with its affine space topology.

If $F$ is a subset of $\mathbb A$ containing an element $x$ in its closure, \textbf{the germ of $F$ in $x$} is the filter $\mathrm{germ}_x(F)$, consisting of all subsets of $\mathbb A$ containing the intersection of $F$ {with} some neighborhood of $x$.

Given a filter $F$ of subsets of $\mathbb A$, its {\bf enclosure} $cl_{\mathbb A}(F)$ (\resp {\bf closure} $\overline F$) is the filter made of  the subsets of $\mathbb A$ containing a set of the form $\cap_{\alpha\in\QF\cup\QF_{im}}\,D(\alpha,k_\alpha)$ that is in $F$, where $k_\alpha\in\QL_\qa\cup\{\infty\}$ (\resp containing the closure $\overline S$ of some $S\in F$). 
Its \textbf{support}, denoted by $supp( F )$, is the smallest affine subspace of $\A$ containing it.

\medskip

A {\bf local face} $F$ in the apartment $\mathbb A$ is associated
 with a point $x\in \mathbb A$ (its vertex) and a  vectorial face $F^v$ in $V$ (its {\bf direction}); it is $F=F^\ell(x,F^v):=germ_x(x+F^v)$.
 Its closure is $\overline{F^\ell}(x,F^v)=germ_x(x+\overline{F^v})$. 
 The enclosure  $cl_{\mathbb A}(F)$ is called a closed face.

There is an order on the local faces: the assertions ``$F$ is a face of $F'$ '',
``$F'$ covers $F$ '' and ``$F\leq F'$ '' are by definition equivalent to
$F\subset\overline{F'}$.
 The {\bf dimension} of a local face $F$ is the dimension of 
  its support; if $F=F^\ell(x,F^v)$, then we have $supp(F):=x+supp(F^v)$.


 A {\bf local chamber}  is a maximal local face, \ie a local face $F^\ell(x,\pm w.C^v_f)$ for some $x\in\A$ and $w\in W^v$.
 The {\bf fundamental local chamber} is $C_0^+=germ_0(C^v_f)$.  A {\bf local panel}  is a local face $F^\ell(x,F^v)$, where $x\in\A$ and $F^v$ is a vectorial panel.


\medskip
 A {\bf sector} in $\mathbb A$ is a $V-$translate $Q=x+C^v$ of a vectorial chamber
$C^v=\pm w.C^v_f$ ($w \in W^v$), $x$ is its {\bf base point} and $C^v$ is its  {\bf direction}; it is open in $\A$.
Two sectors have the same direction if, and only if, they are conjugate
by a $V-$translation, and if, and only if, their intersection contains another sector.

 The {\bf sector-germ} of a sector $Q=x+C^v$ in $\mathbb A$ is the filter $germ(Q)$ of
subsets of~$\mathbb A$ consisting of the sets containing a $V-$translation of $Q$, it is well
determined by the direction $C^v$. So, the set of translation classes of sectors in $\mathbb A$, the set of vectorial chambers in $V$ and the set of sector-germs in $\mathbb A$ are in canonical bijection.

 A {\bf sector-face} in $\mathbb A$ is a $V-$translation $\mathfrak f=x+F^v$ of a vectorial face
$F^v=\pm w.F^v(J)$, its {\bf direction}; its {\bf type} is $J$.
The {\bf sector-face-germ} of $\mathfrak f$ is the filter $\mathfrak F$ of
subsets containing a shortening of $\mathfrak f$ \ie a translation $\mathfrak f'$ of $\mathfrak f$ by a vector in $F^v$ (\ie $\mathfrak f'\subset \mathfrak f$). If $F^v$ is spherical, then $\mathfrak f$ and $\mathfrak F$ are also called
spherical. The sign of $\mathfrak f$ and $\mathfrak F$ is the sign of $F^v$.
We say that $\mathfrak f$ (\resp $\mathfrak F$) is a {\bf sector-panel} (\resp {\bf sector-panel-germ}) if, and only if, $F^v$ is a vectorial panel.

\medskip
A {\bf chimney} in $\mathbb A$ is associated with a local face $F=F^\ell(x, F_0^v)$, called its basis, and to a vectorial face $F^v$, its {\bf direction}; it is the filter
$$
\mathfrak r(F,F^v) := cl_{\mathbb A}(F+F^v).
$$ A chimney $\mathfrak r = \mathfrak r(F,F^v)$ is {\bf splayed} if $F^v$ is spherical, it is {\bf solid} if the direction of its support (as a filter) has a finite  pointwise stabilizer in $W^v$.
A splayed chimney is therefore solid. The enclosure of a sector-face $\mathfrak f=x+F^v$ is a chimney.
The germ of the chimney $\g r$ is the filter $\mathrm{germ}(\mathfrak r )=\mathfrak R$ consisting of all subsets of $\mathbb A$ which contain $ \mathfrak r(F+\qx,F^v)$  for some $\qx\in F^v$.

 \par A  ray $\delta$ with origin in $x$ and containing $y\not=x$ (or the interval $(x,y]$, or the segment $[x,y]$) is called {\bf preordered} if $x\leq y$ or $y\leq x$, {\bf generic} if $x\stackrel{o}{<} y$ or $y\stackrel{o}{<} x$ and {\bf strongly regular} if moreover $y-x$ is in no wall-direction; its sign is $+$ if $x\leq y$ and $-$ if $x\geq y$ (it may be $+$ and $-$, \eg if $W^v$ is finite).
 The germ of this ray is the filter $\mathrm{germ}(\qd)$ consisting of all subsets of $\mathbb A$ which contain a shortening of $\qd$ \ie $\qd\setminus[x,z)$  for some $z\in\qd$.

 \subsection{The masure}\label{1.3}

 In this section, we recall the definition and some properties of a masure  given  in Rousseau~\cite{Rou11}, where the masure is called ``masure affine ordonn\'ee''.

\subsubsection{The apartment system}\label{apart-syst}
An apartment of type $\mathbb A$ is a set $A$ endowed with a set $Isom^w\!(\mathbb A,A)$ of bijections (called \textbf{Weyl--isomorphisms}) such that, if $f_0\in Isom^w\!(\mathbb A,A)$, then $f\in Isom^w\!(\mathbb A,A)$ if, and only if, there exists $w\in W^a$ satisfying $f = f_0\circ w$.
An \textbf{isomorphism} (\resp a \textbf{Weyl--isomorphism}, \resp a \textbf{vectorially Weyl--isomorphism}) between two apartments $\varphi :A\to A'$ is a bijection such that, for any $f\in Isom^w\!(\mathbb A,A)$ and $f'\in Isom^w\!(\mathbb A,A')$, we have $f'^{-1}\circ\qf\circ f\in \Aut(\A)$ (\resp $\in W^a$, \resp $\in(W^v\ltimes V)\cap \Aut(\A)$).
We write $Isom(A,A')$ (\resp $Isom^w(A,A')$, \resp $Isom^w_\R(A,A'))$ for the set of all these isomorphisms.
All isomorphisms considered in \cite{Rou11} are Weyl--isomorphisms.

 As the filters in $\A$ defined in Section~\ref{suse:Faces} (\eg local faces, sectors, walls, etc) are permuted by $\Aut(\A)$, they are well defined in any apartment of type $\A$ and preserved by any isomorphism; each such apartment is also endowed with a canonical structure of real affine space.

\begin{definition}
\label{de:AffineHovel}
A \textbf{masure} (or \textbf{affine ordered hovel})  \textbf{of type $\mathbb{A}$} is a set $\QD$ endowed with a covering $\mathcal{A}$ of subsets  called apartments such that:
\begin{enumerate}
\item[{\bf (MA1)}] any $A\in \mathcal A$ admits a structure of an apartment of type $\mathbb A$;

\item[{\bf (MA2)}] if $F$ is a point, a germ of a preordered interval, a generic  ray or a solid chimney in an apartment $A$ and if $A'$ is another apartment containing $F$, then $A\cap A'$ contains the
enclosure $cl_A(F)$ of $F$ and there exists a Weyl--isomorphism from $A$ onto $A'$ fixing (pointwise) $cl_A(F)$;

\item[{\bf (MA3)}] if $\mathfrak R$ is  the germ of a splayed chimney and if $F$ is a closed face or a germ of a solid chimney, then there exists an apartment that contains $\mathfrak R$ and $F$;

\item[{\bf (MA4)}] if two apartments $A,A'$ contain $\mathfrak R$ and $F$ as in {\bf (MA3)}, then their intersection contains $cl_A(\mathfrak R\cup F)$ and there exists a Weyl--isomorphism from $A$ onto $A'$ fixing (pointwise) $cl_A(\mathfrak R\cup F)$;

\item[{\bf (MAO)}] if $x,y$ are two points contained in two apartments $A$ and $A'$, and if $x\leq_A y$ then the two line segments $[x,y]_A$ and $[x,y]_{A'}$ are equal.
\end{enumerate}
\end{definition}

\par The covering $\mathcal{A}$ appearing in the definition of the masure is called the \textbf{apartment system} of the masure $\QD$.

\par We say that $\QD$ is \textbf{thick} (\resp \textbf{of finite thickness}) if the number of local chambers  containing a given (local) panel is $\geq 3$ (\resp finite).
If $\QD$ is thick, then any wall $M$ in $\QD$ is thick: there are three half-apartments $H_1,H_2,H_3$ such that, for $i\neq j$, $H_i\cap H_j=M$ (and then $H_i\cup H_j\in\sha$)~\cite[Prop. 2.9]{Rou11}.
 For affine buildings the definition of thickness is the same as the one stated above for a masure. The proof of the thickness of any wall is then classical and well known for the discrete affine buildings; for $\R-$buildings that proof is an easy consequence of condition (CO) of Parreau~\cite[1.21]{Parr00}.

  \par An \textbf{automorphism} (\resp a \textbf{Weyl--automorphism}, \resp \textbf{a vectorially Weyl--automorphism}) of $\QD$ is a bijection $\qf:\QD\to\QD$ such that $A\in\sha\iff \qf(A)\in\sha$ and then $\qf\vert_A:A\to\qf(A)$ is an isomorphism (\resp a Weyl--isomorphism, \resp a vectorially Weyl--isomorphism). All isomorphisms considered in \cite{Rou11} are Weyl--isomorphisms.

  \medskip
\par For $x,y\in \QD$, we introduce the relation:

\par $x\leq y$ if, and only if, there is an (or for any) apartment $A$ such that $x,y\in A$ and $x\leq_Ay$ (\ie $f^{-1}(x)\leq f^{-1}(y)$ for any $f\in Isom^w(\A,A)$).

\par This relation is a \textbf{preorder} on $\QD$, invariant by any vectorially Weyl--automorphism.
It is trivial if $W^v$ is finite.

\begin{remark}
\label{rem::NB}
 1) When $W^v$ is finite, a masure of type $\A$ is the same thing as a (not necessarily discrete) affine building of type $\A$ that is {\bf chimney friendly} (\ie any two chimney-germs are in a same apartment).
In particular, the buildings associated with valued root data are masures. Actually by Charignon~\cite{Cha}, a discrete affine building is chimney friendly if and only if it is {\bf sector friendly} (\ie any two sector-germs are in a same apartment). In this article the $\R-$buildings are considered among affine buildings; they are sector friendly by definition.

\par 2) The system of apartments is given with the masure and may not be ``complete'' (even if $\QD$ is a building).
Actually in Rousseau~\cite{Rou12} one obtains masures associated with split Kac--Moody groups over fields endowed with a real valuation, possibly not with a ``complete'' apartment system.
Up to now there is no definition of ``complete masures''. We shall use actually a substitute. See Definition \ref{def::locally_complete} below.

\par {3) After the writing of the first version of this article, H\'ebert published a preprint \cite{Heb17} that contains interesting simplifications of the above definition of masures and many improvements, in particular, about intersections of two apartments.}

\par 4) Affine buildings are locally finite if and only if they are discrete and of finite thickness. We shall use these two properties as a substitute to ``local finiteness'' for masures:
\end{remark}

\begin{definition}
\label{def::locally_finite}
Let $\Delta$ be a masure. We say that $\Delta$ is \textbf{locally finite} if $\Delta$ is semi-discrete {(\ie the model apartment $\A$ is semi-discrete)}  and of finite thickness.
\end{definition}

\par The main examples of thick,  semi-discrete (\resp and of finite thickness) masures are provided by the masures  associated with an almost split Kac--Moody group over a field complete with respect to a discrete valuation and with a perfect (\resp finite) residue field, see \cite{GR08}, \cite{Rou12}, \cite{Cha} and \cite{Rou13}.

\medskip


\subsubsection{ The building at infinity}
\label{subsubsec::building_infinity}

\parni{\bf 1)} By (MA3), two spherical sector-faces (or generic rays) are, up to shortening, contained in a common apartment $A$; we say {that} they are \textbf{parallel} if one of them is a $V$--translation of the other one. This does not depend on the choice of $A$ by (MA4) and  parallelism is an equivalence relation (the proof in \cite{Rou11} uses the spherical assumption).
The parallel class $\partial\qd$ of a generic ray $\qd$ is called an \textbf{ideal point} or \textbf{a point at infinity}. The parallel class $\partial\mathfrak f=\bd\g F$ of a spherical sector-face $\mathfrak f$ (or its germ $\g F$)  is called \textbf{an ideal face} or \textbf{a face at infinity} (an \textbf{ideal chamber} if $\mathfrak f$ is a sector and an \textbf{ideal panel} if $\mathfrak f$ is a sector-panel).
The {\bf type} of $\bd\g f$ is the type of $\g f$.
Actually, a chamber at infinity is nothing else than a sector-germ; its type is $S$.

\par We write $\partial\g f\leq \partial\g f'$ if, for suitable choices of $\g f,\g f'$ in their parallel classes, we have $\g f\subset\overline{\g f'}$.
The ordered set of ideal faces of sign $\pm$ is (the set of spherical faces of) the building $\partial\QD_\pm$ with Weyl group $W^v$; these buildings are twinned (see Rousseau~\cite[3.7]{Rou11}).
We write $\partial\QD=\partial\QD_+ \sqcup\partial\QD_-$ and $\Ch(\partial\QD)=\{\mathrm{ideal\ chambers}\}$.
The type (defined above) of an ideal face is the same as its type defined in the building $\bd\QD_\pm$ with Weyl group $W^v$.
We write $face_{S'}(c)$ the ideal face of type $S'$ of an ideal chamber $c$.
When $W^v$ is finite, $\partial\QD_+=\partial\QD_-$ is a spherical building, twinned with itself.

Notice that, when $\Delta$ is a masure and not a building, $|W^v|=\infty$, the dimension of its apartments is at least $2$ and the rank of its building at infinity $\partial \Delta$ is strictly positive. 

 \par We say that an ideal point $\qx$, an ideal face $\phi$, is at infinity of an apartment $A$ (or a wall $h$, a half apartment $H$, ...) if we may write $\qx=\partial\qd$, $\phi=\partial\g f$, with $\qd\subset A$, $\g f\subset A$ (or $\subset h$, $\subset H$, ...).
We write then $\qx\in\bd A$, $\phi\in\bd A$ (or $\bd h$, $\bd H$, ...); actually $\bd A$ is a twin apartment in $\bd\QD$.
 We say that $\qx\in\phi$ if  (for suitable choices of $\qd,\g f$) $\qx=\bd\qd$, $\phi=\bd\g f$ and $\qd\subset\g f$.
 Actually we may look at the set $\vert\bd\QD\vert$ of ideals points as a geometric realization of $\bd\QD$ and the condition $\qx\in\phi$ (\resp $\qx\in\bd A$) describes $\phi$ (\resp $\bd A$) as a subset of $\vert\bd\QD\vert$.
 Actually $\vert \partial A_+\vert$, with its partition in ideal faces, is isomorphic with the quotient of ${\cal T}^\circ$ by $\R_{>0}$.

\medskip
\parni{\bf 2)}  From (MA3) and (MA4), we see that a point $x\in\QD$ and an ideal chamber $c$ (\resp an ideal face $\phi$, \resp an ideal point $\qx$) determine, in the parallel class $c$ (\resp $\phi$, \resp $\qx$), a unique sector (\resp spherical sector face, \resp generic ray $[x,\qx)$) of base point (\resp origin) $x$. To fix the notation, for a point $x \in \QD$ and an ideal chamber $c \in \partial \Delta$ (\resp ideal face $\phi\in\bd\QD$), we denote by $Q_{x,c}$ (\resp $Q_{x,\phi}$) the sector (\resp sector face) in $\Delta$ with base point $x$ that corresponds to the chamber at infinity $c$ (\resp the ideal face at infinity $\phi$).

Notice that if the sector-face germ $\g F$ is not spherical, then there might exist two different apartments in $\Delta$
 that can contain two different sectors of base point $x \in \Delta$ and with faces both parallel to $\g F$ but different.
This is because in the definition of a masure, there is no axiom stating that for a point and a non-spherical sector-face germ $\g F$ every apartment of $\Delta$ containing both of them must contain the enclosure of $x$ and $\g F$.

\par Any vectorially Weyl--automorphism $\qf$ of $\QD$ acts on $\partial\QD$ as a type-preserving automorphism; if it stabilizes an ideal face $\phi$, it fixes any ideal point $\qx\in\phi$.

\medskip
\parni{\bf 3)} Let $c$ be an ideal chamber at infinity of an apartment $A$. By (MA3) and (MA4), for any $x\in\QD$, there is an apartment $A'$ containing $x$ and the sector-germ $germ(Q)$ associated with $c$ and there is a unique Weyl--isomorphism $\qf:A\to A'$ fixing $germ(Q)$.
 So, by the usual arguments, we see that $x\mapsto\qf(x)$ is a well defined map $\qr_{A,c}:\QD\to A$, called the {\bf retraction of $\QD$ onto $A$ with center $c$}.

\begin{definition}
\label{def::locally_complete}
Let $\Delta$ be a  masure . We say that $\Delta$ is \textbf{locally complete} if the apartment system $\sha$ of $\QD$ is locally-complete, \ie if the following holds:
\begin{enumerate}
\item[]
Take any increasing sequence $\{H_n\}_{n\geq 0}$ of half-apartments that are respectively contained in apartments $\{A_n\}_{n \geq 0}\in\sha$.
We suppose that, for an ideal chamber $c\in \Ch(\partial H_0)$, we have $\bigcup_{n\geq0}\,\qr_{A_0,c}(H_n)=A_0$.
Then $\bigcup_{n\geq0}\,H_n$ is an apartment $A\in\sha$.
\end{enumerate}
\end{definition}

Clearly, for trees, the local-completeness is equivalent to the completeness of the apartment system.
The masures of almost split Kac-Moody groups over complete fields are locally complete.

\subsection{Twin buildings}
\label{sec::notation}
Let $X=(X_{+}, X_{-})$ be a twin building of type $(W,S)$ and we consider it as a set of chambers.
We write $\qd$ and $\qd^*$ its distance and codistance.
Recall {the} following piece of notation, classical or issued from Hartnick--K\"{o}hl--Mars~\cite{HKM}.

For every $c \in \Ch(X)$ and every subset $S' \subset S$ we define the \textbf{$S'$-residue of $c$} to be $res_{S'}(c):= \{ d \in \Ch(X) \; \vert \; \delta(c,d) \in \langle s \; \vert \; s \in S' \rangle\}$. The collection of all $S'$-residues in $X$ will be denoted by $Res_{S'}(X)$. A residue $res_{S'}(c)$ is called spherical if the group $W^v(S')=\langle s \; \vert \; s \in S' \rangle$ is finite (\ie $S'$ itself is spherical); then $res_{S'}(c)$ is a spherical building, see \eg \cite[5.30]{AB}.

In the case of $X:=\Ch(\bd\QD)$ as in Section \ref{subsubsec::building_infinity} and if $S'$ is spherical, $res_{S'}(c)$ is the set of all ideal chambers sharing the same ideal face of type $S'$ as the ideal chamber $c$.
We obtain thus a bijection between the set $face_{S'}(\bd\QD)$ of ideal faces of type $S'$ in $\bd\QD$ and $Res_{S'}(X)$.

Let $\mathit{Sph}(S,W^{v}):=\{ J \subset S \; \vert \; J  \text{ spherical} \}$.

The residues of rank one are called $s$-panels and their set is denoted by $Pan_{s}(X) := Res_{\{s\}}(X)$. By convention we write $P_{s}(c)$ instead of $res_{\{s\}}(c)$.

The residues of co-rank one are called $s$-vertices and their set is denoted by $\mathcal{V}_s:= Res_{S\setminus \{s\}} (X)$. We write $vert_s(c):=res_{S\setminus\{s\}}(c)$.
The canonical embedding $$i : X \hookrightarrow \prod_{s \in S} \mathcal{V}_s \; \;  \text{given by} \; \; c \mapsto (res_{S \setminus \{s\}}(c))_{s \in S}$$ sends a maximal simplex to the tuple consisting of its vertices.

Given a spherical residue $R \subseteq X_{\pm}$ and a chamber $c \in X_{\mp} $ there exists a unique chamber $d \in R$ such that $\delta^{*}(c,d)$ is of maximal length in the set $\delta^{*}(c, R)$ \cite[Lemma~5.149]{AB}. This chamber is called the projection of $c$ onto $R$ and  is denoted by $proj^{*}_{R}(c)$.

Given a chamber $c \in X_{\pm}$ and an element $w \in W$ we denote by $E_{w}(c):= \{ d \in X_{\pm} \;  \vert \; \delta (c,d)= w \in W\}$ and $E_{w}^{*}(c):= \{ d \in X_{\mp} \;  \vert \; \delta^{*}(c,d)= w \in W\}$. The sets $E_{\leq w}(c), E_{<w}(c), E_{\leq w}^{*}(c), E_{< w}^{*}(c) $ are defined accordingly, with respect to the strong Bruhat order (also called Bruhat--Chevalley order).

\medskip
We also define $X_{w}:=\{ (c,d) \in (X_{+} \times X_{-}) \cup (X_{-} \times X_{+}) \;  \vert \; \delta^{*}(c,d)=w \}$ for each $w \in W$.

\subsection{Affine buildings in masures}
\label{1.5}

Recall that by Section~\ref{subsubsec::building_infinity}, an ideal face $\phi$  of $\bd\QD$ is a parallel class of spherical sector-faces and we denote by $J \subset S$ its type. Let $\phi'$ be an ideal face opposite $\phi$.

We denote by $\QD(\phi)$ the set of all (spherical) sector-face-germs $germ(\g f)$ with $\g f$ being in the parallel class $\phi$ and by $\QD(\phi,\phi')$ the union of all apartments $A$ in $\sha$, such that $\phi,\phi'\subset\bd A$; we denote $\sha(\phi,\phi')$ the set of all these apartments.
Actually, by (MA4), $\QD(\phi,\phi')$ depends only on the ``convex hull'' of $\phi,\phi'$ in $\bd\QD$ (a wall if $\phi,\phi'$ are panels).

\begin{proposition}
\label{fact:Iphi} By~\cite[4.3]{Rou11}, $\QD(\phi)$ is an affine, chimney friendly building.
\end{proposition}

More explicitly, any apartment $A(\phi)$ of $\QD(\phi)$ is constructed from a (not necessarily unique) apartment $A$ of $\Delta$, with $\phi$ at infinity of $A$, by the formula  $A(\phi)=\{germ(\g f)\in\QD(\phi) \mid \g f \subset A \}$.
By the same formula, a wall of $A(\phi)$ is constructed from a wall of $A$ that contains some $\g f\in\phi$.
The corresponding vectorial Weyl group of $\QD(\phi)$ is isomorphic to the group $W^v(J)$ generated by $J$.

\par The faces  of $\QD(\phi)$ correspond to chimney-germs of the same direction as $\phi$: if $\g R$ is such a chimney-germ, then the associated face is the filter $\shf(\g R)$ consisting of all subsets of $\QD(\phi)$ containing $\{ germ(\g f) \in\QD(\phi) \mid \g f\subset \QS \}$, for some $\QS\in\g R$.
This face is a chamber in $\QD(\phi)$ if the support of $\g R$ in any apartment is equal to that apartment.

The sector-germs of $\QD(\phi)$ correspond to sector-germs $\g R$ in $\QD$ such that $c=\bd\g R\geq\phi$: they are the filters $\shf(\g R)$ defined as above. In particular, if $c \in \Ch(\bd \Delta)$ is such that $\phi$ is a face of $c$, then the chambers at infinity of $\QD(\phi)$ correspond to $res_{J}(c)$ which is a spherical building.

\begin{proposition}
\label{prop1:Iphi}
Let $\Delta$ be a masure and let $\phi,\phi'$ be opposite ideal faces of $\bd\QD$. Then:

(1) For any sector-face-germ $\g F$ of $\Delta$, with $\bd \g F=\phi$, there is an apartment $A\in \sha(\phi,\phi')$ with $\g F\subset A$.

\par (2) The map $e_\phi:\QD(\phi,\phi')\to\QD(\phi)$, $x\mapsto germ(Q_{x,\phi})$ is onto.
For any $A\in\sha(\phi,\phi')$, $e_\phi(A)=A(\phi)$ and the map $e_\phi:A\to A(\phi)$ identifies $A(\phi)$ with the quotient of $A$ by the vector space generated by the direction of $\phi$.

\par (3) Each $A\in \sha(\phi,\phi')$ may be considered with its natural structure of apartment or with its restricted structure $A^\phi$ that is given by the set $\shm^\phi$ of the walls in $A$ that contain an $\g f\in\phi$.

Moreover, the restricted structure $A^\phi$ is associated (up to some identifications) with the subroot system generated by $((\qa_i)_{i\in J},(\qa_i^\vee)_{i\in J})$ with vectorial Weyl group $W^v(J)$. Thus, the restriction of $e_\phi$ to $A^\phi$ is the essentialization map that identifies $A(\phi)$ with the quotient of $A$ by the vector space intersection of the directions of all walls of $A^\phi$.
\end{proposition}

\begin{proof} The assertion (1) is a consequence of (MA3) applied to $\g F$ and a sector-face-germ $\g F'$ with $\bd\g F'=\phi'$. Then (2) and (3) are clear.
\end{proof}

\begin{proposition}\label{prop2:Iphi}
Let $\Delta$ be a masure and let $\phi,\phi'$ be opposite ideal faces of $\bd\QD$. Let $c_1,c_2$ be two ideal chambers in the residue of $\phi$ (\ie $\phi\leq c_1,c_2$) and opposite in this spherical residue (which is a spherical building).

Then there is a unique apartment $A\in\sha(\phi,\phi')$ such that $c_1,c_2\in \Ch(\bd A)$.
\end{proposition}

\begin{proof}As $c_1,c_2$ are opposite in the residue of $\phi$ and $\phi,\phi'$ are opposite, we have that $c_1$ and the projection $c'_2=proj^*_{\phi'}(c_2)$ are opposite.
We choose $A \in \sha$ such that $\bd A$ is the unique twin apartment containing $c_1$ and $c'_2$.
So $\bd A$ contains $c_1,\phi,c'_2,\phi'$ and $c_2=proj^*_\phi(c'_2)$. Notice that the apartment $A$ is entirely determined by $\bd A$.
\end{proof}

\begin{theorem}\label{theo:Iphi} Let $\Delta$ be a masure and let $\phi,\phi'$ be opposite ideal faces of $\bd\QD$. The set $\QD(\phi,\phi')$ is an affine, chimney friendly building, with system of apartments $\{ A^\phi \mid A \in \sha(\phi,\phi') \}$.
The map $e_\phi:\QD(\phi,\phi')\to\QD(\phi)$ is the essentialization map.
If $\QD$ is semi-discrete (\resp thick, \resp of finite thickness), then $\QD(\phi,\phi')$ is discrete  (\resp thick, \resp of finite thickness).
\end{theorem}

\begin{remark}\label{rema:Iphi} When we endow each $A\in \sha(\phi,\phi')$ with its natural structure of apartment coming from $\Delta$, then $\QD(\phi,\phi')$ has to be considered as a masure that  is not thick. In particular, the axioms of the definition of the masure are verified;  actually $\QD(\phi,\phi')$ is an affine building, except that we added in each apartment many useless walls.

\end{remark}

\begin{proof}[Proof of Theorem~\ref{theo:Iphi}] The apartments $A(\phi)$ of $\QD(\phi)$ correspond to apartments $A$ of $\QD$ with $\phi\in\bd A$ and $A(\phi)=\{ \g F =germ(\g f)\in \QD(\phi) \mid \bd\g F\in\bd A \}$.
But many apartments $A$ may give the same $A(\phi)$.
Actually $A(\phi)$ is well determined by two opposite sector-germs in $\QD(\phi)$ and, {by the paragraph before Proposition~\ref{prop1:Iphi}}, these sector-germs are associated with two sector-germs $\g Q_1,\g Q_2$ in $\QD$, \ie to two ideal chambers $c_1,c_2$, which are opposite in the residue of $\phi$.
 So, by Proposition~\ref{prop2:Iphi}, there is one, and only one, apartment $B\in \sha(\phi,\phi')$ such that $c_1,c_2\in\bd B$, \ie $B(\phi)=A(\phi)$.
 This gives a one to one correspondence between $\sha(\phi,\phi')$ and the set of apartments of $\QD(\phi)$.
 As $\QD(\phi)$ is an affine, chimney friendly building, we clearly obtain the theorem from Proposition~\ref{prop1:Iphi} above.
\end{proof}

\begin{corollary}
\label{lem:TreeWalls}
Let $\Delta$ be a masure of dimension $n$.
Let $\sigma, \sigma' \subset \partial \Delta$ be a pair of opposite panels at infinity. We denote by $\QD(\sigma, \sigma')$ the union of all apartments of $\Delta$ whose boundaries contain $\sigma$ and $\sigma'$. Then  $\QD(\sigma, \sigma')$ is a closed convex subset of $\Delta$, which is an extended tree, \ie splits canonically as a product
$$ \QD(\sigma, \sigma') \cong T \times \R^{n-1},$$
where $T=T(\qs,\qs')\simeq\QD(\qs)$ is a  (chimney friendly) $\R-$tree whose ends are canonically in one-to-one correspondence with the elements of the set $\Ch(\sigma)$ of all ideal chambers having $\sigma$ as a panel. Under this isomorphism, the walls of $\Delta$, contained in $ \QD(\sigma, \sigma')$ and containing $\qs,\qs'$ at infinity, correspond to the subsets of the form $\{v\} \times \R^{n-1}$ with $v$ a vertex of $T$.

\par When $\QD$ is semi-discrete (\resp thick, \resp of finite thickness), then the $\R-$tree $T$ is a genuine discrete (\resp thick, \resp  of finite thickness) tree.
\end{corollary}

\begin{remark} With the above hypotheses, $\QD(\qs)$ (\resp $\QD(\qs,\qs')$) is often called a panel tree (\resp a wall (extended) tree).
\end{remark}

\begin{proof}
This is a consequence of Theorem~\ref{theo:Iphi}  or, better, of~\cite[Section 4.6]{Rou11}.
\end{proof}

\begin{remark}
Clearly a masure $\QD$ is locally complete if and only if the apartment system of any such $\R-$tree $T$ is complete.
\end{remark}

\subsection{Retractions of line segments or rays}\label{1.4b}

The somewhat technical content of this subsection will be only needed in the proof of Proposition \ref{TTB1+}.

Let $(\Delta,\sha)$ be a masure. Let $c$ be an ideal chamber at infinity of an apartment $A \in \sha $. It is associated with a vectorial chamber $\vect c$ in the real vector space $\vect A$.
We identify $(\vect A, \vect c)$ with the pair $(V,C^v_f)$ of Section~\ref{1.1}.

\par We consider, in (another) apartment $A'$ of $\QD$, a preordered segment $[x,y]$ or a preordered ray $\qd$ of origin $x$ (and containing some $y\not=x$).
We parametrize $[x,y]$ or $\qd$ by an affine map $\qp_1:[0,1]\to A'$ or $\qp_1:[0,+\infty)\to A'$, with $\qp_1(0)=x$.
  If $c\in \Ch(\bd\QD_+)$ (\resp $c\in \Ch(\bd\QD_-)$) we assume that $\qp_1$ is increasing (\resp decreasing) for the order $\leq $. This means that  $x\leq y$ (\resp $y\leq x$).

\par We consider the retraction $\qr=\qr_{A,c}$ and the image $\qp=\qr\circ\qp_1$ of the path $\qp_1$.

\begin{proposition}
\label{prop::hecke_path}
The path $\qp$ is a Hecke path with respect to $c$, in the sense of~\cite[Def. 5.2]{GR08} or~\cite[1.8]{GR14}. In particular, after the identification of $(\vect A, \vect c)$ with $(V,C^v_f)$, this means that:

The path $\qp$ is piecewise linear. Except in $0$ (\resp$1$ for segments) the left (\resp right) derivatives $\qp'_-(t)$ (\resp $\qp'_+(t)$) are defined and belong to an orbit $W^v.\ql$ of some $\ql\in\overline{C^v_f}$ (the closure of $\vect c$).

More precisely, we define $w_\pm(t)\in W^v$ as the element with the smallest length such that $\qp'_\pm(t)=w_\pm(t).\ql$. Then we have $w_-(t) \leq w_+(t) \leq w_-(t')  \leq w_+(t')$ when $0<t<t'$ (for the Bruhat--Chevalley order in $W^v$).
\end{proposition}

\begin{proof} We may reduce to the case of a segment.
Unfortunately the references above deal with an opposite case: we have to identify $(\vect A, \vect c)$ with $(V,-C^v_f)$ and consider the reverse path $\overline\qp_1$, defined by $\overline\qp_1(t)=\qp_1(1-t)$; so its image under $\qr$ is defined by $\overline\qp(t)=\qp(1-t)$.
From Theorem 6.2 in~\cite{GR08} (see also Proposition~6.1, there), we obtain that $\overline\qp=\qr\circ\overline\qp_1$ is a Hecke path.
Hence $\overline\qp$ is  piecewise linear and we have $\overline\qp'_\pm(t)\in W^v.\qm$ with $\qm\in\overline{C^v_f}=-\overline{\vec c}$\,; so $\qp$ is piecewise linear and $\qp'_\pm(t)=-\overline\qp'_\pm(1-t)\in W^v.\ql$  with $\ql=-\qm\in\overline{\vec c}$.
Moreover, $\overline\qp'_\pm(t)=\overline w_\pm(t).\qm$ (with $\overline w_\pm(t)$ of minimal length) and, by Lemma 5.4 in \lc, we have $\overline w_-(t) \geq \overline w_+(t) \geq \overline w_-(t')  \geq \overline w_+(t')$ for $t<t'$.
So, for $\qp$, we obtain $\qp'_\pm(t)=w_\pm(t).\ql$ with $w_\pm(t)=\overline w_\mp(1-t)$ and the expected inequalities for the $w_\pm(t)$ are satisfied.
\end{proof}

\section{The cone topology on masures}
\label{sec::cone_top_hovel}
Let $\Delta$ be a masure. As in the case of $\CAT(0)$ spaces we would like to define a topology on the realization $|\bd\QD|$ of the boundary $\partial \Delta $  of the masure $\Delta$, which does not depend on the chosen base point.
Recall that a masure is not necessarily a geodesic metric space, therefore, we cannot apply the \cat theory.  Still, we can define a cone topology on the set of all chambers $\Ch(\bd \Delta)$ at infinity of $\Delta$; that is presented in Section~\ref{subsec::d}.  This topology has interesting dynamical properties (see Proposition \ref{prop::dynamics_str_reg} below), as well as other natural properties that are presented in Sections~\ref{subsec::prop_cone_top} and~\ref{subsec::further_prop}.

\subsection{Definition of the cone topology}
\label{subsec::d}

By a chamber at infinity of $\Delta$ we mean the interior of it and we choose, in any chamber $c$ at infinity, an ideal point $\qx_c$, called its \textbf{barycenter}. We ask that any vectorially Weyl--isomorphism $\qf$ between apartments in $\QD$ permutes these barycenters: $\qf(\qx_c)=\qx_{\qf(c)}$.
Moreover, for any two opposite ideal chambers $c$ and $c_{-}$, we impose that the corresponding barycenters $\qx_{c}$ and $\qx_{c_{-}}$ are also opposite. Actually, we choose $\qx_c$ for $c=\bd C^v_f$. Then $\qx_c$ is uniquely and well defined for any $c\in X^\pm$ by the above conditions, as the vectorially Weyl--isomorphisms between apartments in $\QD$ induce the type-preserving isomorphisms between apartments in $\bd\QD$.

\begin{definition}
\label{def::standard_open_neigh_cone_top}
Let $\Delta$ be a masure, $x \in \Delta$ be a point and $c \in  \Ch(\bd \Delta)$ be a chamber at infinity. Let $\xi_c$ be the barycenter of $c$. By the definition of a masure, we know that there exists an apartment $A \subset \Delta$ such that $x \in A$ and $c \in \Ch(\bd A)$.  Consider the
ray $[x, \xi_c) \subset A$ issuing from $x$ and corresponding to the barycenter $\xi_c$ of $c$ (see Section~\ref{subsubsec::building_infinity}.2) above). Let $r\in [x, \xi_c)$; we define the following subset of $\Ch(\bd \Delta)$
$$
U_{x, r, c}:= \{ c' \in \Ch(\bd \Delta) \; \vert \; [x, r]
\subset [x, \xi_{c'}) \cap [x, \xi_{c})\}.
$$

The subset $U_{x, r, c}$ is called a \textbf{standard open neighborhood} in $\Ch(\partial \QD)$ of the (open) chamber $c$ with base point $x$ and gate $r$.
\end{definition}

\begin{definition}
\label{def::cone_top_on_chambers}
Let $\Delta$ be a masure and let $x$ be a point of $\Delta$. The \textbf{cone topology} $\Top_{x}( \Ch(\bd \Delta))$ on $\Ch(\bd \Delta)$, with base point $x$, is the topology generated by the standard open neighborhoods $U_{x, r, c}$, with $c \in \Ch(\bd \QD)$ and $r\in [x, \xi_c)$.
\end{definition}

\begin{lemma}\label{lem:3.3} The cone topology $\Top_{x}( \Ch(\bd \Delta))$ does not depend on the choice of the barycenters.
\end{lemma}

\begin{proof} Let us consider another choice $\eta$ for the family of barycenters.
Then if  $c'\in  U_{x, r, c}^\qx$ for $r$ far, then $Q_{x,c}\cap Q_{x,c'}\supset cl([x,r])$, which is a ``big'' part of $Q_{x,c}$. {Recall $Q_{x,c}$ is the sector based at $x$ and containing $c$.}
In particular $Q_{x,c}\cap Q_{x,c'}$ contains some far $r'\in [x,\eta_c)$.
By \cite[5.4]{Rou11}, there is a Weyl isomorphism of apartments $\psi$ that sends $Q_{x,c}$ to $ Q_{x,c'}$ and fixes $cl([x,r])$.
Then $\psi$ sends $\eta_{c'}$ to $\eta_c$ and $[x,\eta_{c'})$ to $[x,\eta_c)$.
So $r'\in [x,\eta_{c'})$ and $c'\in  U_{x, r', c}^\eta$.
We have proved that the cone topology associated with the family $\qx$ is finer than the one associated with $\eta$. Now, by symmetry, the result follows.
\end{proof}

\begin{remark}\label{rem:3.3}
When the masure $\QD$ is a (chimney friendly) affine building, $\Delta$ is endowed with a Euclidean metric and it is well known that it is a CAT(0) space. Then the boundary $\bd\QD$ defined in Section~\ref{subsubsec::building_infinity} is contained in the visual boundary $\bd_{\infty}\QD$ of $\QD$ as defined in~\cite{BH99} (it may be different if the apartment system is not complete).
The cone topology of Definition~\ref{def::cone_top_on_chambers} clearly coincides with the restriction on $\Ch(\bd \Delta)$ of the cone topology defined in \cite{BH99} on $\bd_{\infty}\QD$.

 \end{remark}

\begin{proposition}
\label{prop::independence_base_point}

Let $\Delta$ be a masure and let $x,y \in \Delta$ be two different points. Then the cone topologies $\Top_{x}( \Ch(\bd \Delta))$ and $\Top_{y}( \Ch(\bd \Delta))$ are the same.
 \end{proposition}

\begin{proof}
To prove that $\Top_{x}( \Ch(\bd \Delta))$ and $\Top_{y}( \Ch(\bd \Delta))$ are the same, it is enough to show that the  identity map $ \Id : (\Ch(\bd \Delta), \Top_{x}(\Ch(\bd \Delta))) \to (\Ch(\bd \Delta), \Top_{y}(\Ch(\bd \Delta))) $ is continuous with respect to the corresponding topologies. For this it is enough to prove that for every chamber $c \in \Ch(\bd \Delta)$ and every standard open neighborhood $V$ of $c$ in $\Top_{y}(\Ch(\bd \Delta))$, there exists a standard open neighborhood $W$ of $c$ in $\Top_{x}(\Ch(\bd \Delta))$ such that $W \subset V$, as subsets of $\Ch(\bd \Delta)$.

Let us fix a chamber $c \in \Ch(\bd \Delta)$ and let $V:=U_{y, r, c}$ be a standard open neighborhood of $c$ with respect to the base point $y$, where $r\neq y$.

\par Let $Q_{z,c}$  be the sector with base point $z \in \Delta$ corresponding to the chamber at infinity $c \in \partial \Delta$. Notice that, for any two points $z_1,z_2 \in \Delta$, the intersection $Q_{z_1,c} \cap Q_{z_2,c}$ is not empty and moreover, it contains a subsector $Q_{z_3, c}$, where $z_3$ is a point in the interior of $Q_{z_1,c} \cap Q_{z_2,c}$. Take $z_1:=x, z_2:=y$ and $z_3=:z$, such that $z$ is in the interior of the intersection $Q_{y,c} \cap Q_{x,c}$.

To $y$ and $z$ as above apply Lemma~\ref{lem::independence_base_point}, that is stated below. We obtain the existence of a standard open neighborhood $U_{z, r', c}$ of $c$ with base point $z$ such that $U_{z, r', c} \subset U_{y, r, c}$.
 Now for $R:=r' \in [z, \xi_{c} )$ far enough, apply Lemma~\ref{lem::independence_base_point} to the points $z$ and $x$. We obtain the existence of a standard open neighborhood  $U_{x, R', c}$ of $c$ with base point $x$ such that $U_{x, R', c} \subset U_{z, r', c}=U_{z, R, c}$. From here we have that $U_{x, R', c} \subset U_{y, r, c}$ and the conclusion follows.
\end{proof}

\begin{lemma}
\label{lem::independence_base_point}
Let $\Delta$ be a masure and let $c \in \Ch(\bd \Delta)$ be a chamber at infinity. Let $z_1,z_2 \in \Delta$ be two different points such that the sector $Q_{z_1,c}$ contains the point $z_2$ in its interior. Then, for every standard open neighborhood $U_{z_1,r,c}$ of $c$, with base point $z_1$, there exists a standard open neighborhood $U_{z_2,r',c}$ of $c$ with base point $z_2$ such that $U_{z_2,r',c} \subset U_{z_1,r,c}$. And vice versa, for every standard open neighborhood $U_{z_2,R,c}$ of $c$ with base point $z_2$ there exists a standard open neighborhood $U_{z_1,R',c}$ of $c$ with base point $z_1$ such that $U_{z_1,R',c} \subset U_{z_2,R,c}$.
\end{lemma}

\begin{proof}
Let $A$ be an apartment of $\Delta$ such that $Q_{z_1,c} \subset A$. Let $\xi_c$ be the barycenter of the chamber at infinity $c$ and consider the standard open neighborhoods $U_{z_1,r,c}$ and $U_{z_2,R,c}$.

\medskip
Our first claim is that there exists $r' \in (z_2,\xi_c)$ far enough, such that the enclosure $cl_{A}([z_1,r'])$ contains the segment $[z_1, r)$ and also that there exists $R' \in (z_1, \xi_c)$, far enough, such that $cl_{A}([z_2,R']) \supset [z_2,R]$.  Indeed, this is true because for every point $x$ in the interior of $Q_{z_1,c}$ we have that $cl_{A}([z_1,x]) \supset (Q_{z_1,c} \cap Q_{x,c_{-}})$, where $c_{-} \in \Ch(\bd A)$ denotes the chamber opposite $c$ and $[z_1,x]$ is the geodesic segment with respect to the affine space $A$. By taking $r'$, and respectively, $R'$ sufficiently far enough, the claim follows.

\medskip
Let us prove the first assertion of the Lemma. Let $r' \in (z_2,\xi_c)$ satisfying the above claim and let $r'' \in (z_1,\xi_c)$ such that $[z_1,r'']= Q_{z_1,c} \cap Q_{r',c_{-}} \cap [z_1,\xi_c)$. This implies that $U_{z_1,r'',c} \subset U_{z_1,r,c}$. Next we want to prove that in fact $U_{z_2,r',c}\subset U_{z_1,r'',c} \subset U_{z_1,r,c}$.

Let $c'\in U_{z_2,r',c}$; so $[z_2,r'] \subset [z_2,\xi_{c'})\cap[z_2,\xi_c)$.
We may apply Lemma~\ref{lem::rays_line} (see below) to $\delta_2=[z_2,\xi_{c'})$ and $\delta_1=[r',\xi_{c_{-}})$.
We get an apartment $B$ containing these two rays, hence, their enclosures $Q_{z_2,c'}$ and $Q_{r',c_{-}}$. Thus, $c_{-},c' \in \Ch(\partial B)$, and  $B$ contains $cl_A(z_1,r')\supset Q_{z_1,c}\cap Q_{r',c_{-}}$.
In $B$ all rays of direction $\xi_{c'}$ are parallel to $[z_2,\xi_{c'})$, hence to $[z_2,r']$; or $[z_1,r'']\subset A\cap B$.
So $[z_1,r'']\subset[z_1,\xi_{c'})$ and $c'\in U_{z_1,r'',c} $.
This concludes that $U_{z_2,r',c}\subset U_{z_1,r'',c} \subset U_{z_1,r,c}$.

Let us prove the second assertion. From our first claim there exists $R' \in (z_1, \xi_c)$, far
enough, such that $cl_{A}([z_1,R']) \supset (Q_{z_1,c} \cap Q_{R',c_{-}}) \supset [z_2,R]$.
If $c''\in U_{z_1,R',c}$, we apply Lemma~\ref{lem::rays_line} to $\delta_2=[z_1,\xi_{c'})$ and $\delta_1=[R',\xi_{c_{-}})$ and one easily proves that $c''\in U_{z_2,R,c} $, hence $U_{z_1,R',c} \subset U_{z_2,R,c}$.
\end{proof}

\begin{lemma}
\label{lem::rays_line}
Let $\delta_1,\delta_2$ be two preordered rays in apartments $A_1,A_2$ of a masure $\Delta$, with origins $x_1,x_2$. Suppose $x_1\not=x_2$ and $x_1,x_2\in \delta_1\cap\delta_2$ (hence $[x_1,x_2]\subset \delta_1\cap\delta_2$). Then $ \delta_1\cup\delta_2$ is a line in an apartment $A$ of $\Delta$.
\end{lemma}

\begin{proof} This is exactly what is proved in part 2) of the proof of Rousseau~\cite[Prop. 5.4]{Rou11}.
\end{proof}

\begin{lemma} \label{3.6} Let $Q_1,Q_2$ be two sectors in a masure $\QD$, sharing the same base point $x$, with $Q_1$ (\resp  $Q_2$) of positive (\resp negative) direction.
We suppose $Q_1,Q_2$ are opposite, \ie there exist $y_1\in Q_1$, $y_2\in Q_2$ (hence $y_2\stackrel{o}{<} x\stackrel{o}{<} y_1$) such that $x\in[y_1,y_2]$.
Then there is an apartment $A$ containing $Q_1$ and $Q_2$.
\end{lemma}

\begin{remark}
\label{rem::cond_CO}
This is condition (CO) of Parreau~\cite[1.12]{Parr00}
\end{remark}

\begin{proof} Let $\qd'_i$ be the generic ray of origin $x$ in $Q_i$ containing $y_i$.
By (MA3) there is an apartment $A_i$ containing $germ(\qd'_i)$ and $germ_x([x,y_{3-i}])$, hence also $\qd'_i$ (by convexity) and some point $x_i\in (x,y_{3-i}]$.
Now it is clear that $\qd_i=\qd'_i\cup[x,x_i]$ is a (generic) ray in $A_i$.
Applying Lemma \ref{lem::rays_line}, there is an apartment $A$ containing $\qd_1$ and $\qd_2$.
Now, by (MA2) and (MA3), $A$ contains $cl_A(\qd'_i)\supset Q_i$.
\end{proof}

\begin{remark}\label{rem:3.9}
\begin{enumerate}
\item[1)]

Let $x$ be a fixed base point in $\Delta$ and let $S'\subset S$ be spherical. With the $S'-$residue $res_{S'}(c)$ of an ideal chamber $c \in \Ch(\bd\QD)$ there is associated in a unique way the ideal face of type $S'$ of $c$. Therefore, on the set $Res_{S'}(\Ch(\bd\QD))$ of all $S'-$residues in $\Ch(\bd\QD)$ one can define a cone topology (with base point $x$)  in the same way as in Definition~\ref{def::cone_top_on_chambers}.
More precisely,  in such an ideal face $\phi$ of type $S'$ (that corresponds to a unique $S'-$residue $res_{S'}(c)$) one can choose an ideal point $\qx_\phi$ (equivariantly under vectorially Weyl--isomorphisms between apartments in $\QD$) and then the definition of the standard open neighborhoods $U_{x, r, \phi}$ of $\phi$ in $Res_{S'}(\Ch(\bd\QD))$ goes as in Definition~\ref{def::standard_open_neigh_cone_top}.
 In Proposition~\ref{top.face} below we prove that this cone topology on $Res_{S'}(\Ch(\bd\QD))$, with $S'$ spherical, does not depend on the chosen base point $x$.

\item[2)]
Suppose $\QD' \subset \QD$ is a ``sub-masure'' of $\Delta$, this means that $\QD'$ is union of some apartments of $\QD$ that also satisfy the axioms of Definition~\ref{de:AffineHovel}.
 Then $\Ch(\bd\QD')\subset \Ch(\bd\QD)$. It is clear that the cone topology on $\Ch(\bd\QD')$ viewed as a masure itself, coincides with the restriction on $\Ch(\bd\QD')$ of the cone topology on $\Ch(\bd\QD)$ (assuming that the base point $x$ is in $\QD'$).

For example this is the case when $\QD'=\QD(\phi,\phi')$ (as in Theorem \ref{theo:Iphi}), with apartments endowed with their natural structure of apartments coming from $\Delta$ (see Remark~\ref{rema:Iphi}).

\par Let us now consider the affine building $\QD''=\QD(\phi,\phi')$, where any apartment $A\in\sha(\phi,\phi')$ is endowed with its restricted structure of apartment $A^\phi$. Because there are more directions of walls in $A$ than in $A^\phi$, any ideal chamber $c''_0$ in $A^\phi$ contains an ideal chamber $c_0$ in $A$ and we choose $\qx_{c''_0}:=\qx_{c_0}$.
Extending these choices by invariance under vectorially Weyl--isomorphisms between apartments of $\QD''$, we obtain an injection $\qi:\Ch(\bd\QD'')\into \Ch(\bd\QD')\subset \Ch(\bd\QD)$ and a choice of $\qx_{c''}\in c''$ for any $c''\in \Ch(\bd\QD'')$, such that $\qx_{c''}=\qx_{\qi(c'')}$. (As there are less walls in $A^\phi$, an ideal chamber in $A^\phi$ may contain two (or more) chambers of $A$.)

By the above property of $\qi$, one has $\qi^{-1}(U_{x,r,\qi(c'')}^{\bd\QD})=U_{x,r,c''}^{\bd\QD''}$; thus, the cone topology on $\Ch(\bd\QD')$ (or $\Ch(\bd\QD)$)  induces the same topology as the cone topology on $\Ch(\bd\QD'')$ (where $\QD''$ is an affine building, hence  a \cat space, as explained in Remark~\ref{rem:3.3}).
 By Lemma~\ref{lem:3.3} the initial choice of $c_0$, \ie of $\qx_{c''_0}=\qx_{c_0}$, has no influence on the topology.

\par Recall that (by Proposition~\ref{fact:Iphi} and Theorem~\ref{theo:Iphi}), for an ideal chamber $c$ with $\phi\leq c$ and $J$ the type of $\phi$, the set $\Ch(\bd\QD'')$ is the residue $res_J(c)$.
 \end{enumerate}
 \end{remark}

\subsection{Basic properties of the cone topology}
\label{subsec::prop_cone_top}

We want to verify basic properties of the cone topology on $\Ch(\partial \Delta)$, where $\Delta$ is a masure. Many of these properties are inspired from Kramer~\cite{K02} and Hartnick--K\"{o}hl--Mars~\cite{HKM}. Recall also the notation of Section~\ref{sec::notation}.

In what follows we let $X:= \Ch(\partial \Delta)$. {Recall $E_{1}^{*}(c):= \{ d \in X_{\mp} \;  \vert \; \delta^{*}(c,d)= 1 \in W\}$.}

\begin{proposition}
\label{lem::E1_open_set}
Let $\Delta$ be a masure. For each $c \in X_{\pm}$ the set $E_{1}^{*}(c)$ is an open subset of X.
\end{proposition}
\begin{proof}
To prove the lemma we recall that the cone topology on $X$ does not depend on the chosen base point. Therefore, let $d$ be a chamber in $E_{1}^{*}(c)$, that is to say opposite $c$. We need to prove that there is an open neighborhood of $d$ in the cone topology that is a subset of $E_{1}^{*}(c)$.  For this let $A$ be the unique apartment of $\Delta$ containing $c,d$ in $\Ch(\partial A)$ and fix a base point $x \in A$ for the cone topology.  For  every $r \in (x, \xi_{d})$, every chamber in $U_{x,r,d}$ is opposite $c$ by Lemma~\ref{3.6}. The proposition is proven.
\end{proof}

\begin{corollary}\label{rem_TTB4+} Let $\QD$ be a masure and $X=\Ch(\bd\QD)$. Then any residue $res_J(c)\subset X$ is closed for the cone topology of $X$.
\end{corollary}
\begin{proof} This is deduced from Proposition \ref{lem::E1_open_set} by the same twin-building-arguments as in~\cite[Lemma 3.4]{HKM}.
\end{proof}

\begin{lemma}
\label{TTB1}
Let $\Delta$ be a masure. Then the cone topology on $X$ is Hausdorff.
\end{lemma}
\begin{proof}
 This is a direct consequence of the definition of the cone topology on $X$.
\end{proof}

\begin{proposition}
\label{TTB2}
Let $\Delta$ be a masure. For each $J \in \mathit{Sph}(S,W^{v})$ the map $$p_{J}: X_{1} \to X_{+} \cup X_{-}$$  $$(c,d) \mapsto proj_{res_{J}(c)}^{*}(d)$$ is continuous, where $X_{1}:=\{ (c,d) \in (X_{+} \times X_{-}) \cup (X_{-} \times X_{+}) \;  \vert \; \delta^{*}(c,d)=1 \}$.
\end{proposition}

\begin{proof}
Notice that $X_{+} \times X_{-}$ and  $X_{-} \times X_{+}$ are endowed with the product topology from $X_{+}$  and $X_{-}$.
Therefore, the topology on $X_1$ is the following. Let $c,d \in X$ such that $\delta^{*}(c,d)=1$. A standard open neighborhood of $(c,d)$ in $X_1$ is given by $X_1 \cap (U_{x,r,c} \times U_{x,r',d})$, where $U_{x,r,c}$ and $U_{x,r',d}$ are standard open neighborhoods of $c$, respectively $d$, with $r \in [x, \xi_c)$, $ r' \in [x, \xi_d)$ and $x$ is a base point for the cone topology on $X$. Actually by taking $r \in (x, \xi_c)$, $ r' \in (x, \xi_d)$ and $x$ in a same apartment as $c$ and $d$, then, by Lemma \ref{3.6}, $ (U_{x,r,c} \times U_{x,r',d})\subset X_1$.

Let $d':=proj_{res_{J}(c)}^{*}(d)$, where $(c,d) \in X_1$. Let $A$ be the unique apartment in $\Delta$ such that $c,d,d' \in \Ch(\bd A)$ and fix a base point $x \in A$. Recall that the cone topology on $X$ does not depend on the chosen base point.

We need to prove that given a standard open neighborhood $U_{x,r,d'}$ of $d'$ there exist open neighborhoods $U_{x,r_1,c}, U_{x,r_2,d} $ of $c$ and respectively $d$ such that the image of $(U_{x,r_1,c}\times  U_{x,r_2,d})$ by the map $p_J$ is included in $U_{x,r,d'}$. For $r_1,r_2$ far enough, $cl_A([r_1,r_2])\supset[x,r]$.
If $c_1\in U_{x,r_1,c}$ and $d_1\in U_{x,r_2,d}$, there is an apartment $A_1$ containing $Q_{x,c_1}\supset[x,r_1]$ and $Q_{x,d_1}\supset[x,r_2]$ (Lemma \ref{3.6}).
Now, by \cite[Prop. 5.4]{Rou11}, there is a Weyl isomorphism $\qf:A\to A_1$ fixing $cl_A([r_1,r_2])$.
Clearly $\qf$ sends $d$ to $d_1$, $c$ to $c_1$, $res_J(c)$ to $res_J(c_1)$, hence $d'=p_J(c,d)$ to $d'_1=p_J(c_1,d_1)$ and $[x,\qx_{d'})$ to $[x,\qx_{d'_1})$.
But $\qf$ fixes $cl_A([r_1,r_2])\supset[x,r]$, so $[x,r]\subset [x,\qx_{d'})\cap[x,\qx_{d'_1})$ and $d'_1=p_J(c_1,d_1)\in U_{x,r,d'}$.
\end{proof}

\begin{corollary}
\label{TTB1_2}
Let $\Delta$ be a masure. For each $s \in S$ and each $c \in X_{\pm}$ the map $$proj_{P_{s}(c)}^{*}: E_{1}^{*}(c) \to P_s(c)\subset X_{\pm}$$
$$d \mapsto proj_{P_{s}(c)}^{*}(d)$$ is continuous, where $E_{1}^{*}(c), P_{s}(c)$ are endowed with the induced topology from the cone topology on $X$.

Moreover, if $c'\in E_1^*(c)$, then the {restriction of the} map $proj_{P_{s}(c)}^{*}: P_s(c') \to P_s(c)$ is a homeomorphism.
\end{corollary}

\begin{proof} By Proposition~\ref{lem::E1_open_set}, recall that $E_{1}^{*}(c)$ is open in $X$ and by Corollary~\ref{rem_TTB4+}, $P_{s}(c)$ is closed in $X$. Then the first part of the corollary is a consequence of Proposition~\ref{TTB2} by replacing $res_{J}(c)$ with $P_{s}(c)$, for $s  \in S$.
The final assertion is then proved as in~\cite[Prop. 3.5]{HKM}, using some twin-building-arguments.
\end{proof}

\begin{lemma}
\label{essential} Let $\QD$ be a masure. We suppose the existence of a group $V_0$ of vectorially Weyl--automorphisms of $\QD$ that stabilizes each apartment $A\in\sha$, induces on each such apartment $A$ a group of translations $\vect{V_0}$ and stabilizes each wall in $A$.
We assume moreover that $\QD/V_0$ has a (natural) structure of a masure with apartments $\{A/\vect{V_0} \mid A\in \sha\}$ and set of walls  in each $A/\vect{V_0}$ the set $\{M/\vect{V_0} \mid M\in \shm_A\}$.

\par Then there is a natural homeomorphism $\Ch(\bd\QD)\to \Ch(\bd(\QD/ V_0))$ with respect to the cone topology on $\Ch(\bd\QD)$ and the cone topology on $\Ch(\bd(\QD/ V_0))$.
\end{lemma}

\begin{proof} We denote $\qp:\QD\to\QD/V_0$ to be the quotient map.
Clearly, each sector $Q$ in $\QD$ is stable under $V_0$ and $\qp(Q)=Q/V_0$ is a sector in $\QD/V_0$.
So we obtain a bijection $\qp:\Ch(\bd\QD)\to  \Ch(\bd(\QD/V_0)), \; germ(Q)\mapsto germ(\qp(Q))$.
We have now to identify the corresponding  two cone topologies. 

Let $x$ be a base point in $\Delta$.

\par For $c=germ(Q)\subset A$, then $\qx_c$ is the class of a generic ray $\qd$ whose direction $\vect\qd$ is not in $\vect{V_0}$; so $\qp(\vect\qd)$ is the direction of a generic ray in $\qp(A)=A/V_0$, whose class we denote by $\qx_{\qp(c)}$.
For $c\in \Ch(\bd\QD)$, $\qp([x,\qx_c))$ is the ray $[\qp(x),\qx_{\qp(c)})$.
So, for $c,c'\in \Ch(\bd\QD)$, $\qp([x,\qx_c)\cap[x,\qx_{c'}))\subset[\qp(x),\qx_{\qp(c)})\cap[\qp(x),\qx_{\qp(c')})$.
This proves that the map $\qp:\Ch(\bd\QD)\to  \Ch(\bd(\QD/V_0))$ is continuous.

\par Conversely, let $c \in \Ch(\bd \Delta)$ and  suppose $[\qp(x),\qx_{\qp(c)})\cap[\qp(x),\qx_{\qp(c')})$ is big, for $c' \in \Ch(\bd \Delta)$.
We choose an apartment $A$ in $\QD$ containing $\{x\}\cup c$ and write $c_1$ the ideal chamber opposite $c$ in $A$.
The retraction $\qr_{A,c_1}$ sends $c'$ to $c$ (by Lemma~\ref{3.6}), hence $\qx_{c'}$ to $\qx_c$ and $[x,\qx_{c'})$ to $[x,\qx_c)$.
Moreover it fixes $\qp^{-1}([\qp(x),\qx_{\qp(c)})\cap[\qp(x),\qx_{\qp(c')}))$.
So $[x,\qx_{c'})\cap[x,\qx_c)$ is big: we have proved that the map $\qp^{-1}:  \Ch(\bd(\QD/V_0))\to \Ch(\bd\QD)$ is continuous.
\end{proof}

\begin{remark}[\textbf{Proof of Proposition \ref{prop:cone_top_residue}}]
\label{rem::essential} When $\QD$ is an affine (extended) building, the hypotheses of Lemma~\ref{essential} are clearly satisfied {where $V_0$ is the group of automorphisms that are} inducing on each apartment $A$ the group of translations $\vect{V_0}$ that is determined by the vector space intersection of the directions of all walls in $A$.
Therefore, Lemma~\ref{essential} applies in the situation of Theorem~\ref{theo:Iphi}.
By Remark~\ref{rem:3.9}.2), Theorem~\ref{theo:Iphi} and this Lemma \ref{essential}, the cone topology on $\Ch(\bd\QD(\phi))=\Ch(\bd\QD(\phi,\phi'))$ coincides with the topology induced by the cone topology of $\Ch(\bd\QD)$. As a conclusion, when $\QD$ is a masure, the induced cone topology on any spherical residue in $\Ch(\bd\QD)$, which is a spherical building, is described classically through an affine (essential) building. This concludes the proof of Proposition \ref{prop:cone_top_residue}.
\end{remark}

\begin{proposition}
\label{TTB4+}
Let $\Delta$ be a masure.  Assume that $\QD$ is locally finite and locally complete. Then for each $s \in S$, every panel $P \in Pan_{s}(X_{\pm})$ is compact, with respect to the induced cone topology.
\end{proposition}

\begin{proof} Let $\qs\subset\bd\QD$ be the panel at infinity corresponding to $P$, \ie $P$ is the set $\Ch(\qs)$ of all ideal chambers having $\qs$ as a panel.
We choose an opposite panel $\qs'\subset\bd\QD$.
By Corollary~\ref{lem:TreeWalls} there is in $\QD$ an extended tree $\Delta(\qs,\qs')\simeq T\times\R^{n-1}$, where $T$ is an $\R-$tree whose set of ends $\QO$ is in one to one correspondence with $\Ch(\qs)$.
As $\QD$ is locally finite, $T$ is a genuine (\ie discrete) tree and, as $\QD$ is locally complete, the apartment system of $T$ is complete; so $\QO$ is compact.
We have seen, by Lemma~\ref{essential} and Remark~\ref{rem::essential}, that the topologies on $\QO=\Ch(\bd T)$ and on $\Ch(\bd \Delta(\qs,\qs'))$ coincide. Therefore, the residue $P$ is compact.
\end{proof}

\begin{corollary}
\label{TTB4}
Let $\Delta$ be a masure.  Assume that $\QD$ is locally finite and locally complete. Then for each $s \in S$ there exists a compact panel $P \in Pan_{s}(X_{\pm})$.
\end{corollary}

\begin{proof}
This follows from Proposition~\ref{TTB4+}.
\end{proof}

\begin{corollary}\label{Schub.compact}
Let $\Delta$ be a masure and $c\in \Ch(\bd\QD)$.  Assume that $\QD$ is locally finite and locally complete. Then for any $w\in W^v$, the Schubert variety $E_{\leq w}(c)$ is compact for the cone topology (induced by the cone topology of $\Ch(\bd\QD)$). In particular, any spherical residue is compact.
\end{corollary}

\begin{proof} This is a consequence of Propositions~\ref{lem::E1_open_set},~\ref{TTB4+}, Corollary~\ref{TTB1_2} and Lemma~\ref{TTB1}.
One uses a Bott--Samelson--Demazure desingularization of $E_{\leq w}(c)$ involving galleries in $\Ch(\bd\QD)$, see~\cite[Section 3.3]{HKM}.
\end{proof}

\subsection{Further properties of the cone topology}
\label{subsec::further_prop}

\begin{lemma}
\label{TTB5}
Let $\Delta$ be a masure. The set $X_1:=\{ (c,d) \in (X_{+} \times X_{-}) \cup (X_{-} \times X_{+}) \;  \vert \; \delta^{*}(c,d)=1 \}$ is open with respect to the induced cone topology of $X$.
\end{lemma}

\begin{proof}
As in the proof of Proposition~\ref{TTB2}, we have that $(U_{x,r,c} \times U_{x,r',d}) \subset X_1$ is an open neighborhood of $(c,d) \in X_1$, where $x \in A$, $A$ is the apartment of $\Delta$ with $c,d \in \Ch(A)$ and $r,r'$ are far enough.
\end{proof}

 \begin{proposition}
\label{TTB6}
Let $\Delta$ be a masure. For every $S' \subset S$, the canonical map $res_{S'}: X_{\pm} \to Res_{S'}(X_{\pm})$ is open with respect to the quotient topology on $Res_{S'}(X_{\pm})$ and the cone topology on $X_{\pm}$.

In particular, for every $s \in S$ the canonical map $vert_s: X_{\pm} \to \mathcal{V}_s^{\pm}$ is open with respect to the quotient topology on $\mathcal{V}_s^{\pm}$ and the cone topology on $X_{\pm}$.
\end{proposition}

\begin{proof}
The quotient topology on $ Res_{S'}(X_{\pm})$ (\resp $\mathcal{V}_s^{\pm}$) is the finest such that the canonical map $res_{S'}: X_{\pm} \to Res_{S'}(X_{\pm})$ (\resp $vert_s: X_{\pm} \to \mathcal{V}_s^{\pm}$) is continuous, \ie a subset $U \subset Res_{S'}(X_{\pm}) \; (\text{\resp }\mathcal{V}_s^{\pm})$ is open if and only if $res_{S'}^{-1}(U))$ (\resp $vert_s^{-1}(U))$ is open in $X_{\pm}$ with respect to the cone topology on $X$.

Fix a base point $x \in \Delta$. It is enough to prove that for a standard open neighborhood $U_{x,r,c}$ of $c \in X$ the set $res_{S'}^{-1}(res_{S'}(U_{x,r,c}))$ is open in $X_{\pm}$. Let $d \in res_{S'}^{-1}(res_{S'}(U_{x,r,c}))$, then there exists $c' \in U_{x,r,c}$ such that $c'$ and $d$ share the same facet of type $S'$. Let $A$ be an apartment that contains $c'$ and $d$ in its boundary $\Ch(\partial A)$.
Notice that $x$ is not necessarily a point in $A$. Take $y$ in $A$ as a base point (recall that the cone topology does not depend on the chosen base point) and choose a standard open neighborhood $U_{y,r',c'}$ such that $U_{y,r',c'} \subset U_{x,r,c}=U_{x,r,c'}$.
We claim that there exists a standard open neighborhood $U_{y, r_d,d}$ of $d$ such that  $U_{y, r_d,d} \subset res_{S'}^{-1}(res_{S'}(U_{x,r,c}))$.
 In order to prove the claim it is enough to prove that there exists $r_d \in [x, \xi_d)$ far enough such that for every $d' \in U_{y, r_d,d}$ there exists $c'' \in U_{y,r',c'}$ with the property that $c''$ and $d'$ share the same facet of type $S'$. Indeed, let $\qs$ be the facet of type $S'$ shared by $c'$ and $d$ in $\bd A$. We write $d_1\subset\bd A$ the chamber opposite $d$. We take $r_d\in [y,\qx_d)$ far enough such that $cl_A(r_d,d_1)$ contains $[y,r']$ and an open neighbourhood of $y$ in $A$. For $d'\in U_{y,r_d,d}$, the sectors $Q_{y,d_1}$ and $Q_{y,d'}$ are opposite, hence in a same apartment $B$, by Lemma \ref{3.6}. Now $B$ contains $[y,r_d]$, $cl_A(r_d,d_1)\supset[y,r']$ and an open neighbourhood of $y$ in $A$.
 We consider the retraction $\qr=\qr_{B,d_1}$ and $c''=\qr(c')$.
 Then $[y,\qx_{c''})=\qr([y,\qx_{c'}))\supset [y,r']$, hence $c''\in U_{y,r',c'}$.
 Moreover $c',d$ in $\bd A$, share the same facet of type $S'$.
 So $c''=\qr(c'),d'=\qr(d)$ in $\bd B$, share the same facet of type $S'$.
\end{proof}

\begin{proposition}
\label{top.face}
Let $\QD$ be a masure and let $S'\in Sph(S,W^v)$. We identify $Res_{S'}(X_\pm)$ with the set $face_{S'}(\bd\QD_\pm)$ of all ideal faces of type $S'$ in $\bd\QD_\pm$ as in Section \ref{sec::notation}. So the map $res_{S'}:X_\pm\to Res_{S'}(X_\pm)$ is  identified with the map $face_{S'}$ that sends any ideal chamber to its face of type $S'$.

\par Then the quotient topology on $Res_{S'}(X_\pm)$ defined in Proposition~\ref{TTB6} coincides with the topology on $face_{S'}(\bd\QD_\pm)$ defined in Remark \ref{rem:3.9}.1). In particular, the latter topology does not depend on the base point $x$.
\end{proposition}

\begin{proof} We fix a base point $x$ for the topology on $face_{S'}(\bd\QD_\pm)$. We have to prove that the map $face_{S'}:X_\pm\to face_{S'}(\bd\QD_\pm)$ is continuous and open. Furthermore, both $Res_{S'}(X_\pm)$ and $face_{S'}(\bd\QD_\pm)$ are endowed with the quotient topologies of the cone topology on $X_\pm$ by the surjective map $res_{S'}$ identified with $face_{S'}$.

Indeed, let $U_{x,r,\phi}=\{ \phi'\in face_{S'}(X_\pm) \mid [x,r]\subset[x,\qx_{\phi'})\cap[x,\qx_{\phi}) \}$ be an open neighborhood of $\phi=face_{S'}(c)$, with $c\in X_\pm$.
 For $r'$ far  in $[x,\qx_c)$, we have $cl([x,r'])\supset[x,r]$.
 So, for $c'\in U_{x,r',c}$, the closed sector  $\overline{Q_{x,c'}}$ contains  $[x,r]$ and, for $\phi'=face_{S'}(c')$, $[x,r]\subset [x,\qx_{\phi'})$, thus $\phi'\in U_{x,r,\phi}$. We have proved that the map $face_{S'}$ is continuous.

 \par Let us prove now that the map $face_{S'}$ is open, \ie for any $c\in X_\pm$ and $r\in [x,\qx_c)$, $face_{S'}(U_{x,r,c})$ is open.
 Indeed, for any $c'\in U_{x,r,c}$ and $\phi=face_{S'}(c')$, we have to find $r'\in[x,\qx_\phi)$ such that $U_{x,r',\phi}\subset face_{S'}(U_{x,r,c})$; this means that for every $ \phi'\in U_{x,r',\phi}$ there exists $c''\in U_{x,r,c'}=U_{x,r,c}$ with $face_{S'}(c'')=\phi'$.
 We choose an apartment $A$ containing $Q_{x,c'}$, hence $A$ contains $Q_{x,\phi}$. We define $\phi_-$ to be the ideal face in $A$ opposite $\phi$ and $c_-=proj^*_{\phi_-}(c')$ an ideal chamber in $A$.
 We choose $r'\in[x,\qx_\phi)$ far enough, such that $[x,r]\subset Q_{r',c_-}$. This is possible as, in $A$, $cl(Q_{x,\phi}\cup Q_{x,c_-})=cl(Q_{x,\phi_-}\cup Q_{x,c'})$. Indeed,  in $\bd A$, $c_-=proj^*_{\phi_-}(c')$ is on the same side as $c'$ with respect to any (twin) wall $M$ containing $\phi$ and $\phi_-$ (otherwise the reflection $r_M$ would increase the codistance); so in $\bd A$, $cl(\phi\cup c_-)\subset cl(\phi_-\cup c')$ and there is equality, as $c'=proj^*_\phi(c_-)$.
 This equality is sufficient to have that $[x,r]\subset Q_{r',c_-}$.

 By the following Lemma \ref{opp.faces}, there is an apartment $B$ containing $Q_{x,c_-}$ and $Q_{x,\phi'}$ (for $\phi'\in U_{x,r',\phi}$), hence also $Q_{r',c_-}\supset [x,r]$.
 We consider the retraction $\qr=\qr_{B,c_-}$ and $c''=\qr(c')\subset B$.
 Clearly $\qr$ sends $A$ to $B$ (isomorphically), $[x,\qx_{c'})$ to $[x,\qx_{c''})$, $\phi$ to $\phi'$ and fixes $A\cap B\supset  Q_{r',c_-}\supset [x,r]$.
 So $face_{S'}(c'')=\qr(face_{S'}(c'))=\qr(\phi)=\phi'$ and $[x,\qx_{c'})\cap [x,\qx_{c''})\supset [x,r]$; we have $c''\in U_{x,r,c'}$ as expected.
\end{proof}

\begin{lemma}
\label{opp.faces}
Let $\QD$ be a masure, $A$ an apartment in $\QD$, $x$ a point in $A$ and $\phi,\phi_-$ opposite ideal faces in $A$.
We consider $c_-\in \Ch(\bd\QD)$ with $ \phi_-\subseteq c_-$.
Then, for any $\phi'\in U_{x,r',\phi}$ with $r'\in(x,\qx_\phi)$, there exists an apartment $B$ containing $Q_{x,c_-}$ and $Q_{x,\phi'}$.
\end{lemma}

\begin{proof} We consider the local chamber $F^\ell: =germ_x(Q_{x,c_-})$.
By axiom (MA3) there is an apartment $A_1$ containing the closed local chamber $\overline{F^\ell}$ and also $\phi'$, hence $A_1$ contains $Q_{x,\phi'}$. Thus $A\cap A_1$ contains $cl(x,r')$.
We consider in $A_1$ the sector $Q_{x,c'}$, with base point $x$ and opposite $F^\ell$; it is thus opposite $Q_{x,c_-}$ and $\overline{Q_{x,c'}}\supset Q_{x,\phi'}$.
By Lemma~\ref{3.6} there is an apartment $B$ containing $Q_{x,c_-}$ and $Q_{x,c'}$, hence also $Q_{x,\phi'}$.
\end{proof}

\begin{proposition}
\label{TTB1+}
Let $\Delta$ be a masure and $S'\subset S$, with $S'\not=\emptyset$, $S'\not=S$.
Then the set $Res_{S'}^\pm:=Res_{S'}(X_\pm))$ is Hausdorff with respect to the quotient topology induced from the cone topology on $X_{\pm}$.

In particular, the  vertex sets $\mathcal{V}_s^{\pm}$, $s \in S$, are Hausdorff with respect to the quotient topology.
\end{proposition}

\begin{remark}
\label{rem::TTB1+}
When $S'$ is spherical, this is also a consequence of Proposition~\ref{top.face}:  as in Lemma~\ref{TTB1}, we see easily that the topology on $face_{S'}(\bd\QD_\pm)$ is Hausdorff.
\end{remark}

\begin{proof}
The quotient topology on  $Res_{S'}^\pm$   is the finest such that the canonical map $res_{S'}: X_{\pm} \to Res_{S'}^\pm$ is continuous, \ie a subset $U$ of $Res_{S'}^\pm$ is open if and only if $res_{S'}^{-1}(U)$ is open in $X_{\pm}$ with respect to the cone topology on $X$.

We prove that  $Res_{S'}^\pm$ is Hausdorff.
Indeed, take $\qs_1 \neq \qs_2 \in Res_{S'}^\pm$ and consider $c_i \in \qs_i \subset X$, for $i \in \{1,2\}$.
We have that $c_1 \neq c_2$ and that $\qs_i = res_{S'}(c_i)$, for $i \in \{1,2\}$ (using the notation in Section \ref{sec::notation}).
Consider an apartment $A \in \sha$ such that $c_1,c_2 \in \Ch(\partial A)$ and take a point $x \in A$ as a base point for the cone topology on $X$. Let $w:= \delta(c_1, c_2)$. As $c_1$ and $c_2$ do not share their face of type $S'$, we have that $w \notin W^v(S')$. Thus $w$ contains some $s\in S\setminus S'$ in any reduced decomposition. Moreover, every $w'$ with $w' \geq w$, with respect to the Bruhat--Chevalley order, does contain $s$ in any reduced decomposition.

Take $U:= U_{x,r,c_1}$ and $V:= U_{x,r',c_2}$.
Notice that the open sets $res_{S'} (U)$, $res_{S'} (V)$ are open neighborhoods of $\qs_1$, respectively $\qs_2$, by Proposition~\ref{TTB6} and also that
$$res_{S'}^{-1}(res_{S'} (U)) \text{ and } res_{S'}^{-1}(res_{S'} (V))$$
are open subsets in $X$, by the definition of the quotient topology on $Res_{S'}^\pm$.

Let $c_{-} \in \Ch(\partial A)$ be the chamber opposite $c_1$. Without loss of generality, we consider $r\in [x,\qx_{c_1})$
 far enough such that $cl_A(r,c_{-})$ contains $[x,r']$.

We need to prove that $res_{S'} (U) \cap res_{S'} (V) = \emptyset $. Suppose the contrary,  that $res_{S'} (U) \cap res_{S'} (V) \neq \emptyset$. Thus, there exist an ideal chamber $d_1 \in U$ and an ideal chamber $d_2 \in V$ that share their face of type $S'$. We claim that this cannot be possible.

Indeed, for $d_1 \in U= U_{x,r,c_1}$, the sectors $Q_{x,c_-}$ and $Q_{x,d_1}$ are opposite, hence in a same apartment $B$, by Lemma \ref{3.6}. Now $B$ contains $[x,r]$ and $cl_A(r,c_-)$, which contains $[x,r']$ if $r$ is far enough. Applying the retraction $\rho_{B, c_{-}}$, we have $\rho_{B, c_{-}}(c_1)=d_1$ and $\delta(d_1, \rho_{B, c_{-}}(c_2))= w$.
Now by Proposition~\ref{prop::hecke_path} applied  to the preordered ray $[x,\xi_{d_2})$ and the retraction $\rho_{B, d_1}$ we have that, for $t$ great, $\qd(d_1,d_2)=w_\pm(t)\geq w_+(0)=\delta(d_1, \rho_{B, c_{-}}(c_2))= w$.
Notice that for $\qp_1(t) \in [x,r')$ we have that $w_+(t) =w$.
Therefore, we have obtained that $ \delta(d_1, d_2)$ contains $s$ in any reduced decomposition, thus $d_1$ and $d_2$ cannot share a face of type $S'\subset S\setminus\{s\}$. This is a contradiction of our assumption and the conclusion follows.
\end{proof}

\begin{proposition}
\label{TTB7}
Let $\Delta$ be a masure and $\QS\subset Sph(S,W^v)$ satisfying $\bigcap_{J\in\QS}\,J=\emptyset$. Then the diagonal embedding $\iota : X_{\pm} \hookrightarrow \prod\limits_{J \in \QS} Res_J(X_{\pm})$ given by $c \mapsto (res_{J}(c))_{J \in \QS}$ is a homeomorphism onto its image, with respect to the cone topology on $X$.
\end{proposition}

\begin{proof}
The topology on $\prod\limits_{J \in \QS} Res_J(X_{\pm})$ is the product topology given by the topology on every $Res_J(X_{\pm})$. Recall that the topology on $Res_J(X_{\pm})$ is the finest topology such that the canonical map $res_J: X_{\pm} \to Res_J(X_{\pm})$ is continuous, \ie a subset $U$ of $Res_J(X_{\pm})$ is open if and only if $res_J^{-1}(U)$ is open in $X_{\pm}$ with respect to the cone topology on $X$.

Notice that the map $\iota$ is injective and continuous.
Indeed, by the hypothesis on $\QS$, the union of all faces (of some chamber) having the type in $\QS$ is in no wall; so this chamber  is uniquely determined by its faces having the type in $\QS$ and we have that $\iota$ is injective.
The continuity follows from the fact that the canonical maps $res_J: X_{\pm} \to Res_J(X_{\pm})$ are continuous, for every $J \in \QS$.

To prove that $\iota$ is a homeomorphism onto its image it is enough to prove that $\iota$ is an open map onto its image.

Indeed, to prove this it is sufficient to show that, for every standard open neighborhood $U$ of $X_{\pm}$ with respect to the cone topology, the set $\iota(U)$ is open in the image $\iota(X_{\pm})$ that is endowed with the topology induced from the topology on $\prod\limits_{J \in \QS} Res_J(X_{\pm})$.

Let $c \in X_{\pm}$, $x \in \Delta$ and $r \in (x, \xi_{c})$. Denote by $A$ an apartment containing $x$ and containing the chamber $c$ at infinity. Let $c_-$ be the chamber opposite $c$ in $\partial A$. We have $\iota(U_{x,r,c})= \{(res_{J}(d))_{J \in \QS} \; \vert \; d \in U_{x,r,c}\} \subset \prod\limits_{J \in \QS} Res_{J}(X_{\pm})$. Take $(res_{J}(d))_{J \in \QS}$ for some $d \in U_{x,r,c}$. A standard open neighborhood of $(res_{J}(d))_{J \in \QS}$ in $ \prod\limits_{J \in \QS} Res_{J}(X_{\pm})$ is of the form $\prod\limits_{J \in \QS} res_J(U_{x,r_{J},d})$, where $r_J \in (x, \xi_{d})$ for every $J \in \QS$.
Notice that $r_J$ might be different from $r_{J'}$, for $J\neq J'$. It is sufficient to prove that there exists $r_d \in (x, \xi_d)$ such that $\left( \prod\limits_{J \in \QS} res_J(U_{x,r_{d},d}) \right) \cap \iota(X_{\pm}) \subset \iota(U_{x,r,c})$, \ie for every $d' \in X_{\pm}$ such that there exists $d_J \in U_{x,r_d,d}$ with $res_{J}(d')=res_{J}(d_J)$, for every $J \in \QS$, we have that $d' \in U_{x,r, c}$.

Consider $r_d$  such that $r \in (x, r_d)$. We can suppose that because $d \in U_{x,r,c}$. Let $d' \in X_{\pm}$ such that there exists $d_J \in U_{x,r_d,d}$ with $res_{J}(d')=res_{J}(d_J)$, for every $J \in \QS$.


Let $A'$ be the apartment of $\Delta$ containing $c_-$ and $d$ at infinity. Actually, by Lemma \ref{3.6}, the closed sectors $\overline{Q_{x,c_-}}$ and $\overline{Q_{x,d}}$ are opposite and in $A'$.


\par Recall that for a spherical face $\phi$ at infinity of $\Delta$ and a point $z \in \Delta$ there exists a uniquely defined sector face $Q_{z,\phi}$ with base point $z$ and ideal face $\phi$. This might not be the case when $J \subset S$ is not spherical.

 For $J\in\QS$, let $\phi_J$ (\resp $\phi'_J$) be the ideal face of type $J$ of $d$ (\resp $d'$ or $d_J$).
 As $d_J\in U_{x,r_d,d}$, Proposition 5.4 in~\cite{Rou11} tells that there is a Weyl--isomorphism of apartments of $\Delta$ sending $\overline{Q_{x,d}}$ onto $\overline{Q_{x,d_J}}$, and fixing $cl_A(x,r_d)$.
 So, if $r_d$ is far enough, $\overline{Q_{x,\phi_J}}\cap \overline{Q_{x,\phi'_J}}$ is big in $\overline{Q_{x,\phi_J}}$ (\ie for any Euclidean metric on $A'$, the apartment $A'$ contains any given bounded subset of $\overline{Q_{x,\phi_J}}$).


 Now $\overline{Q_{x,d'}}$ contains each $\overline{Q_{x,\phi'_J}}$, so it contains the big subset $\overline{Q_{x,\phi_J}}\cap \overline{Q_{x,\phi'_J}}$ of $\overline{Q_{x,\phi_J}}$.
 By the hypothesis on $\QS$, the union $\bigcup_{J\in\QS}\,Q_{x,\phi_J}\subset \overline{Q_{x,d}}$ is in no wall; so the convex hull of $\bigcup_{J\in\QS}\,(\overline{Q_{x,\phi_J}}\cap \overline{Q_{x,\phi'_J}})$
  is big in $\overline{Q_{x,d}}$.
 By the axiom (MAO), this implies that $\overline{Q_{x,d}} \cap \overline{Q_{x,d'}}$ is big in $\overline{Q_{x,d}}$.
So  $d' \in U_{x,r,c}$ , for $r_d$ far enough.
\end{proof}

\begin{remarks}
\label{rem::TTB7} 1) Notice that a natural choice for $\QS$ is the set $\QS_p:=\{\,\{s\} \mid s\in S\}$, the corresponding residues being the panels.

\par 2) Another natural choice would be the set $\QS_v: =\{\,S\setminus\{s\} \mid s\in S\}$ with vertices as residues.
By Proposition~\ref{TTB7} this is possible when each $S\setminus\{s\}$ is spherical  (\ie essentially, when $W^v$ is of finite, affine or strictly hyperbolic type).
In this case,
we obtain that  the canonical embedding $$i : X_{\pm} \hookrightarrow \prod_{s \in S} \mathcal{V}_s^{\pm} \; \;  \text{given by} \; \; c \mapsto (vert_{s}(c))_{s \in S}$$is a homeomorphism onto its image, with respect to the cone topology on $X$ and the product topology on $\prod_{s \in S} \mathcal{V}_s^{\pm}$.

\end{remarks}

\subsection{Comparison with other topologies}\label{ss:compar}

In this subsection we compare the properties proved above of the cone topology on $\Ch(\bd \Delta)$, when $\Delta$ is a masure, with the axioms introduced in Kramer~\cite{K02} and Hartnick--K\"{o}hl--Mars~\cite{HKM}, for abstract twin buildings.

Lemma~\ref{TTB1} is axiom (TTB1) from~\cite{HKM}. Proposition~\ref{TTB2} is a generalization of axiom (TTB2) of~\cite{HKM} and the first part of Corollary~\ref{TTB1_2} is exactly that axiom. Proposition~\ref{TTB4+} is axiom (TTB4+) and its Corollary~\ref{TTB4} is axiom (TTB4).
Therefore, when $\QD$ is locally finite and locally complete, $\Ch(\bd\QD)$ is a weak topological twin building in the terminology of \cite{HKM}.

Lemma~\ref{TTB5},  Proposition~\ref{TTB6} and Proposition~\ref{TTB1+} for the case of vertices are respectively, axioms (TTB5), (TTB6), (TTB1+) from~\cite{HKM}. Proposition~\ref{TTB7} is a generalization of axiom (TTB7) from~\cite{HKM} that is stated in  Remark~\ref{rem::TTB7}.2).
Actually (TTB7) is essential to make the link between the point of view of \cite{HKM} and the one of \cite{K02}. One may remark that Kramer \cite{K02} is studying essentially the affine twin building case and in that case, the axiom (TTB7) is proven to be true (see Remark~\ref{rem::TTB7}.2)) in the context of the cone topology.

The definition of topological twin buildings that are studied in \cite{HKM} involves axioms (TTB1), (TTB2), (TTB4) that are presented above and also the following axiom:

(TTB3) There exist chambers $c_{\pm} \in X_{\pm}$ such that $X_{\pm} = \lim\limits_{\to} E_{\leq w} (c_{\pm})$.

The latter mentioned equality is clear set-theoretically, but one wants it topologically. More precisely, we say that a topological space $Y$ is the direct limit of subspaces $Y_i$, denoted by $Y =\lim\limits_{\to} Y_i$, if $Y= \bigcup Y_i$ and $U \subset Y$ is open if and only if $U \cap Y_i$ is open for every $i \in I$.

Following \cite{HKM}, a weak topological twin building (\ie a twin building endowed with a topology satisfying (TTB1), (TTB2), (TTB4)) may be endowed with a ``completed'' topology (\ie the finest topology inducing the given topology on each Schubert variety $E_{\leq w} (c_{\pm})$) and one obtains a topological twin building.

\par Actually the cone topology never satisfies (TTB3) (if $W^v$ is infinite and $\Delta$ is thick). This is  Proposition \ref{propTTB3 non satisfait} in the Appendix by H\'ebert.
 Still, the cone topology studied in this current article is natural and has many good and useful properties, as it was proven in Sections~\ref{subsec::prop_cone_top} and~\ref{subsec::further_prop}. It is not clear whether the associated completed topology (\ie the topological twin building) is better than the cone topology.




\section{Stabilizers and strong transitivity in masures}
\label{s2}
\label{s4}

\par We recall now the definition of strongly transitive actions, for combinatorial buildings as well as for affine $\R-$buildings and twin buildings.

\begin{definition}\label{def:1.2}
Let $\Delta$ be a combinatorial (hence discrete) building (\resp an affine $\R-$ building) and let $\sha$ be the (not necessarily complete) apartment system defining $\Delta$. Let $G \leq \Aut(\Delta)$. We say that $G$ acts strongly transitively on $\Delta$ if $G$ acts transitively on the set  $\sha$ of apartments and if, for one (so every) apartment $A\in\sha$, the stabilizer $\Stab_G(A)$ of $A$ in $G$ acts transitively on the set $\Ch(A)$ of chambers in $A$ (\resp induces on $A$ a group containing the affine Weyl group $W^a$).

When $\Delta= (\Delta_{-}, \Delta_{+})$ is a twin building with apartment system $\sha=(\sha_+, \sha_-)$ and $G \leq \Aut(\Delta)$ we say that $G$ acts strongly transitively on $\Delta$ if $G$ acts transitively on the set of all twin apartments  $A=(A_-, A_+) \in \sha$ and $\Stab_{G}(A)$ acts transitively on $\Ch(A_+)$ (so also on $\Ch(A_-)$), for every twin apartment  $A=(A_-, A_+)$ in $\sha$.
Notice that this last definition is purely combinatorial, hence independent of any choice of a geometric realization of $\QD$.
\end{definition}



We now reproduce the same results as in Caprace--Ciobotaru~\cite[Section 3]{CaCi}, where affine buildings are studied. We only translate those results into the language of masures.

\begin{lemma}
\label{lem:unip}
Let  $\Delta$ be a masure  and let $G \leq \Aut(\Delta)$ be any group of vectorially Weyl--automorphisms. Let $c \in \Ch(\partial \Delta)$ be a chamber at infinity. Then the set
$$
G_c^0 :=\{g \in G_c \; | \;  g \text{ fixes some point of }\Delta\}
$$
is a normal subgroup (called the horospherical stabilizer) of the stabilizer $G_c :=\{g \in G \; | \;  g(c)=c\}$.
\end{lemma}

\begin{proof}
It is clear that $G_c, G_c^0$ are subgroups of $G$, as an element $g \in G_c$ fixing a point $x$ of $\Delta$ will also fix (pointwise) an entire sector that emanates from $x$ and pointing to $c$.  This is also used to prove that $G_c^{0}$ is normal in $G_c$ (one proceeds as in~\cite[Section 3]{CaCi}).
\end{proof}

\begin{lemma}
\label{lem:BusemanRetraction}
Let  $\Delta$ be a masure  and let $G \leq \Aut(\Delta)$ be any group of vectorially Weyl--automorphisms. Let $c \in \Ch(\partial \Delta)$ be a chamber at infinity and $A$ be an apartment of $\Delta$, whose boundary contains $c$. Then for any $g \in G_c$, the map
$$\beta_c(g) \colon A \to A : x \mapsto \rho_{A, c}(g(x))$$
is an automorphism of the apartment $A$, acting as a (possibly trivial) translation.
Clearly $\qb_c(g)$ and $g$ coincide on $A\cap g^{-1}(A)$.
Moreover the map
$$\beta_c \colon G_c \to \Aut(A)$$
is a group homomorphism whose kernel coincides with $G_c^0$.
In particular $[G_c,G_c]\subset G_c^0$.
\end{lemma}

\begin{proof}
For $g\in G_c$, $g(A)$ is an apartment of $\Delta$ that contains the chamber $c$ in its boundary at infinity $\partial g(A)$.
 By the definition of the retraction $\rho_{A, c}$, as $g$ is vectorially Weyl and $c \in \Ch(\partial A) \cap \Ch(\partial g(A))$, it is easy to see that indeed $\beta_c(g) $ is an element of $\Aut(A)$, that can act as a translation or  can fix a point of $A$.

Let us prove that $\beta_c$ is a group homomorphism, whose kernel coincides with $G_c^0$. Let $g, h \in G_c$. Because $c \in \Ch(\partial A) \cap \Ch(\partial g(A)) \cap Ch(\partial h(A))$, there exists a common sector $Q_{x,c}$ in $A\cap g(A)\cap h(A)$. This implies that for $\beta_c$ to be a group homomorphism it is enough to consider a point $y$ far away in the interior of the sector  $Q_{x,c}$ such that $ \rho_{A,c}(h(y))= h(y)$. We have that $$\beta_c(g) \beta_c(h) (y)= \beta_c(g) (\rho_{A, c}(h(y)))= \beta_c(g)(h(y))=\rho_{A, c}(g(h(y)))=\beta_c(gh)(y).$$

It is clear from the definition that $G_c^0$ is contained in the kernel of $\beta_c$. Let now $g \in \Ker(\beta_c)$. This implies that $g$ fixes a point in the intersection $A \cap g(A)$; therefore, $g \in G_{c}^{0}$. The conclusion follows.
\end{proof}

\begin{lemma}
\label{lem:Levi}
Let $\Delta$ be a masure  and let $G \leq \Aut(\Delta)$ be a subgroup of vectorially Weyl--automorphisms that acts strongly transitively on $\partial \Delta$.

Then for any pair $c, c'$ of opposite chambers at infinity, every $G_c^0$--orbit on the set $c^{op}$ of chambers opposite $c$ is invariant under $G_{c, c'}:=G_c\cap G_{c'}$.
\end{lemma}

\begin{proof}
Let $c, c'$ be a pair of opposite chambers at infinity of $\Delta$ and let $d\in c^{op}$. As $G$ acts strongly transitively on $\partial \Delta$ there is some $g \in G_c$ with $d = g(c')$. As the quotient $G_c/G^0_c$ is abelian by Lemma~\ref{lem:BusemanRetraction}, we have that the subgroup $H:=G_{c, c'}G_{c}^{0} $ is normal in $G_c$ and
$$\begin{array}{rcl}
G_{c, c'}(G^0_c(d)) & = & H(g(c')) \\
&= & g(H(c'))\\
& = & g(G^0_c.G_{c, c'}(c'))\\
& = & gG^0_c(c')\\
& = & G^0_c(g(c'))\\
&=& G^0_c(d).
\end{array}
$$
Thus the $ G^0_c$--orbit of $d$ is indeed invariant by $G_{c, c'}$, as claimed. It follows that the $H$--orbits on $c^{op}$ coincide with the $G^0_c$--orbits.
\end{proof}

In this section we also verify an analogue criterion for strong transitivity as in Caprace--Ciobotaru~\cite[Section 3]{CaCi}. We first give a definition.

\begin{definition}
\label{def::good_part}
Let $\Delta$ be a masure  and let $G \leq \Aut(\Delta)$ be a group of automorphisms. Let $A$ be an apartment of $\Delta$. We say that a subset (or a filter) $\Omega$ of the apartment $A$ is {\bf $G-$friendly}  if, for any two apartments $A,A'$ containing $\QO$, there is $g\in G_\QO^{0}:=\{ g\in G \; \vert \;  g \text{ fixes pointwise } \QO \}$ such that $A'=g(A)$ and $g : A\to A',x\mapsto g(x)$ is a Weyl--isomorphism.
\end{definition}

\medskip
\par For $\QO$ a local chamber or a sector-germ, $\QO$ is $G-$friendly if, and only if, $G_\QO^{0}$ is transitive on the set of all apartments containing $\Omega$: this is a consequence of (MA2) and (MA4) as two isomorphisms $A\to A'$ fixing $\QO$ are necessarily equal; this implies in particular that the corresponding isomorphism is Weyl.  In the same way, we also obtain that $G_\QO^{0}$ fixes pointwise $\overline\QO$ and $cl_A(\QO)$.

\begin{definition}
\label{4.2}
(See Gaussent--Rousseau~\cite[1.5]{GR14} or Rousseau~\cite[4.10]{Rou13}). Let $\Delta$ be a masure  and let $G$ be a subgroup of $\Aut(\Delta)$. We say that $G$ acts {\bf strongly transitively} on $\Delta$ if any isomorphism involved in axioms (MA2), (MA4) may be chosen to be the restriction of an element of $G$.
\end{definition}

\par This means that each of the sets $cl_A(F)$ appearing in (MA2) and $cl_A(\g R\cup F)$ appearing in (MA4) is $G-$friendly.

\par The Kac--Moody groups as in the example of Section~\ref{apart-syst} act strongly transitively by vectorially Weyl--automorphisms on the corresponding masures, see~\cite{GR08} and~\cite{Rou12} in the split case and \cite{Rou13} in the general almost split case.

\begin{proposition}
\label{lem:ST:criterion}
Let  $\Delta$ be a 
 masure  and let $G \leq \Aut(\Delta)$ be a group of vectorially Weyl--automorphisms. Then the following conditions are equivalent.

\begin{enumerate}[(i)]
\item $G$ is strongly transitive on $\Delta$;
\item
every local chamber of $\Delta$ is $G-$friendly;
\item
every sector-germ of $\Delta$ is $G-$friendly.
\end{enumerate}
\end{proposition}

\begin{remark}\label{5.4}Notice that for thick, affine, chimney friendly buildings (viewed as $\R-$buildings) there are two definitions of strong transitivity: Definition \ref{def:1.2} and Definition \ref{4.2}. By the proof of Proposition \ref{5.5} below these two definitions are the same.

In addition, the proof of Proposition~\ref{lem:ST:criterion} gives the equivalence of the above three conditions for thick, sector friendly  affine $\R-$buildings. For this one needs only condition (CO) of~\cite{Parr00}.
\end{remark}

\begin{proof}[Proof of Proposition~\ref{lem:ST:criterion}]
(i) $\Rightarrow$ (iii): This follows from the definition of the strongly transitive action of $G$ on the masure  $\Delta$.

\medskip
(ii) $\Rightarrow$ (i):
First we claim that the enclosure $cl(F)$ of any local face $F$ of the masure  $\Delta$ is $G-$friendly. Indeed, if there exist two apartments $A$ and $A'$ such that $F \subset A \cap A'$, we consider a local chamber $C$ of $A$ and a local chamber $C'$ of $A'$ such that both cover $F$.
By Rousseau~\cite[Prop. 5.1]{Rou11} there is an apartment $A''$ of $\Delta$ containing $cl(C)$ and $cl(C')$ (both containing $cl(F)$).
 Applying our hypothesis to $(C,A,A'')$ and $(C',A',A'')$, the claim follows.


\medskip
We need to verify that all isomorphisms involved in the definition of the masure  are induced by elements of $G$.  Therefore, let $\widetilde\QO=cl_A(F)$ or $\widetilde\QO=cl_A(\g R\cup F)$ be as in the axiom (MA2) or (MA4). We consider a closed convex subset $\QO\subset supp(\widetilde\QO)$ such that $cl(\QO)$ contains the filter $\widetilde\QO$ and
is contained in the intersection $A \cap A'$ of two apartments $A, A'$ of $\Delta$.
 But in $cl(\Omega)$ one can find a (maximal) local face $F_1$ such that $F_1 \subset cl(\Omega) \subset supp(cl(F_1))$, where $supp(cl(F_1))$ is the unique affine space of minimal dimension that contains $cl(F_1)$.
 For this face $F_1$ we apply our above claim and we obtain an element $g$ in the pointwise stabilizer  $G_{cl(F_1)}^{0} < G$ such that $g(A)=A'$, $g : A\to A',x\mapsto g(x)$ is a Weyl--isomorphism and $g$ fixes pointwise the closed face $cl(F_1)$.
 As $cl(F_1) \subset cl(\Omega) \subset supp(cl(F_1))$ and $g$ is an affine isomorphism from $A$ onto $A'$, we also have that $g$ fixes (pointwise) $cl(\QO)\supset\widetilde\QO$. The conclusion follows.

\medskip
(iii) $\Rightarrow$ (ii):
Let $C$ be a local chamber contained in the intersection of two apartments $A$ and $A'$ of $\Delta$.
 Let $x$ be a vertex of $C$ and consider in $A$ the sector $Q_x$ with base point $x$ that contains the chamber $C$.
 Then $A\cap A'$ contains a neighborhood $C'$ of $x$ in $Q_x$ and so that $C\subset C'$.
 Let $y$ be a point in the interior of $C'$.
 Take the sector in $A$ of the form $Q_y:= Q_x+ (y-x) \subset Q_x$.
 In the apartment $A'$ consider the sector $Q_y'$ with base point $y$  that contains the vertex $x$, hence which is opposite  $Q_y$.
  Notice that $Q'_y \cap Q_x$ is a small neighborhood of $x$ (\resp of $y$) in $Q_x$ (\resp in $Q_y'$).
  By applying 
 Lemma~\ref{3.6}, one concludes that there is an apartment $A''$ containing the sectors $Q_x$ and $Q'_y$.
 So $A''$ contains $Q_x\subset cl_A(\{x\}\cup Q_y)$ and $Q_x\cap Q'_y$ is an element of the filter $C$.
  By the hypothesis, applied successively to $(Q_x,A,A'')$ and $(Q'_y,A',A'')$, we obtain the conclusion.
\end{proof}

\begin{proposition}\label{5.5} 
Let $\QD$ be a {thick} masure, let $X$ be its twin building at infinity and let $G\leq \Aut(\QD)$ be a group of vectorially Weyl automorphisms.
If $G$ acts strongly transitively on $\QD$, then the action of $G$ induced on $X$ is strongly transitive as well.
\end{proposition}

\begin{proof}
By the axiom (MA3), it is clear that $G$ is transitive on the set of all apartments of $(\Delta, \sha)$. It remains to prove that the stabilizer $\Stab_{G}(A)$ of an apartment $A \in \sha$ is transitive on the set $\Ch(\partial A_+)$. To obtain this it is enough to prove that the affine Weyl group $W^{a}$ is contained in $\Stab_{G}(A)$. 

First, consider three half-apartments $H_1,H_2, H_3 \subset \Delta$ that have a wall $h$ in common and such that $H_i \cap H_j=h$, for every $i\neq j$. We claim there exists an element $g \in G$ such that $g$ fixes $H_1$ and $g(H_2)=H_3$. In particular, we obtain that $g(A_2)=A_3$, where $A_2=H_1\cup H_2$ and $A_3=H_1 \cup H_3$. Indeed, let $Q_1$ be a sector in $H_1$ such that it admits a sector-panel-germ at infinity of $h$, which is denoted by $\mathfrak{F}_1$. Consider at infinity of $h$ a sector-panel-germ $\mathfrak{F}_2$ that is opposite $\mathfrak{F}_1$. Let $Q_2$ (\resp $Q_3$) be a sector in $H_2$ (\resp $H_3$) that contains $\mathfrak{F}_2$ at infinity. One notices that $Q_1$ and $Q_2$ (\resp $Q_3$) are of opposite direction in $A_2$ (\resp $A_3$). We apply axiom (MA4) for  $germ(Q_1)$ and $\mathfrak{F}_2$ and to the apartments $A_2$ and $A_3$. Then, there exists a Weyl--isomorphism $g \in G$ such that $g(A_2)=A_3$ and $g$ is fixing pointwise $cl_{A_2}(germ(Q_1),\mathfrak{F}_2)$. Notice that $ H_1 = cl_{A_2}(germ(Q_1),\mathfrak{F}_2)$; therefore, $g(H_1)=H_1$ pointwise. The claim is proved.

Let now $h$ be a wall of the apartment $A$ and denote by $H_1$, $H_2$ the half-apartments of $A$ such that $\partial H_1=\partial H_2=h$. We claim that there exists $g \in  W^{a} \leq \Stab_{G}(A)$ such that $g$ is a reflection with respect to the wall $h$. Indeed, as the masure  is thick, by Rousseau~\cite[Prop.~2.9]{Rou11} there exists a third half-apartment $H_3 \subset \Delta$ such that $H_1 \cap H_3= H_2 \cap H_3$. From the above claim, applied three times, we obtain an element $g \in \Stab_{G}(A)$ with the desired properties. In particular, as $g$ fixes pointwise the wall $h$, we have that $g \in W^{a}$ and the conclusion follows.
\end{proof}

\begin{definition}
\label{def::LST}  
Let $(\Delta, \sha)$ be a masure that is not a tree {(or an extended tree)} and let $G$ be a vectorially Weyl subgroup of $\Aut(\Delta)$.
Let $\sigma, \sigma' \subset \partial \Delta$ be a pair of opposite panels and denote by $\QD(\sigma, \sigma')$ the union of all apartments of $(\Delta, \sha)$ whose boundaries contain $\sigma$ and $\sigma'$.
By Corollary~\ref{lem:TreeWalls} consider the panel tree $T(\sigma, \sigma'):= (T(\sigma, \sigma'), \sha(\sigma, \sigma'))$ associated with $\sigma$ and $\sigma'$.  Let $G_{\sigma, \sigma'}$ be the stabilizer in $G$ of $\sigma, \sigma'$. We say that $G$ satisfies \textbf{condition (LST)} if for each pair of opposite panels $(\sigma, \sigma')$ the subgroup $G_{\sigma, \sigma'}$ acts strongly transitively ({as in  Definition \ref{def:1.2}}) on the panel tree $T(\sigma, \sigma')$.
We shall say that a group acting on an (extended) tree satisfies condition (LST) if it acts strongly transitively on the associated (quotient) tree.
\end{definition}

\begin{lemma}\label{5.7} 
{Let $\Delta$ be a masure (can be a tree or an extended tree).} If $G\leq\Aut(\QD)$  is strongly transitive on $\QD$, then it satisfies Condition (LST).
\end{lemma}

\begin{remark} If we assume the apartment system is ``locally complete'' (see Definition~\ref{def::locally_complete}) and $G$  strongly transitive on $\partial \Delta$, then it is possible to see that  $G$ satisfies (LST).
Without this completeness assumption, we obtain only that $G_{\sigma, \sigma'}$ is $2-$transitive on the ends of $T(\qs,\qs')$ which is not sufficient for (LST) (see Remark~\ref{rem::DTE_not_ST}).
\end{remark}

\begin{proof} The ends of $T(\qs,\qs')$ correspond to the ideal chambers in $\Ch(\qs)$.
 It is well known (and a consequence of Proposition~\ref{lem:ST:criterion} and Remark~\ref{5.4}  above) that a group $H$ acting on a tree is strongly transitive if, and only if, the horospherical stabilizer $H_e^0$ of an end $e$ of $T(\qs,\qs')$  is transitive on the apartments containing this end $e$.

 To such an end $e$ there are associated a {unique} ideal chamber $c$ in $\Ch(\qs)$ and {its unique projected ideal} chamber $c'$ in $\Ch(\qs')$ {(see Section~\ref{1.5})}. We consider the associated sector germs {$c=germ(Q)$ and $c'=germ(Q')$}.
 As $G$ is strongly transitive on $\QD$, the subgroup $G_{c,c'}^0$ of $G$ fixing pointwise these sector germs is transitive on the apartments containing them (\ie containing $\qs,\qs'$ and $e$).
 But $G_{c,c'}^0 \leq G_{\sigma, \sigma'}$ is in the horospherical stabilizer of $e$ with respect to the subgroup $G_{\sigma, \sigma'}$. So $G_{\sigma, \sigma'}$ is strongly transitive on $T(\qs,\qs')$.
\end{proof}

\section{Existence and dynamics of strongly regular elements}
\label{sec::str_reg_elements}

\begin{definition}(See~\cite[Section~2.1]{CaCi})
\label{def::str_reg_lines}
Let $\Delta$ be a masure  and let $A$ be an apartment of $\Delta$. With respect to the affine structure of $A$, a line $\ell \subset A$ is called \textbf{strongly regular}, if its points at infinity lie in the interior of two opposite chambers of the twin building at infinity of $\Delta$. This also means that both associated ray-germs are strongly regular. In particular, the apartment $A$ is the unique apartment of $\Delta$ containing the strongly regular line $\ell$.

{By definition, $\gamma \in \Aut(\Delta)$ is called \textbf{hyperbolic} if there is an apartment $A$ of $\Delta$ and a translation axis in $A$, \ie a line in $A$ along which $\gamma$ acts like a translation. Note in general $\Aut(\Delta)$ can contain more types of elements than just elliptic, hyperbolic or parabolic (if they exist).} Moreover, a hyperbolic element $\gamma$ of $\Aut(\Delta)$ is called \textbf{strongly regular} if it admits a strongly regular translation axis (\ie there exists an apartment $A $ of $\Delta$ and a strongly regular geodesic line in $A$ which is a translation axis of $\gamma$).
\end{definition}

\begin{theorem}
\label{thm:ExistenceStronglyReg}
Let $(\Delta, \sha)$ be a thick, masure such that $\A^v=(V,W^v)$ has no irreducible factor of affine type.  Let $G$ be a vectorially Weyl subgroup of $\Aut(\Delta)$ that acts transitively on $\sha$  and satisfies condition (LST).
Then $G$ contains a strongly regular hyperbolic element.
\end{theorem}

\begin{remark}
\label{rem::AGT} As stated, the theorem fails when $\QD$ is of affine type: we may consider a split loop group $G$ over a non-Archimedean local field, acting on its masure $\Delta$. It acts strongly transitively on $\partial\QD$, so it acts transitively on the set of all apartments of $\Delta$, but the smallest positive imaginary root $\qd\in\QF^+_{im}$ is trivial on the maximal torus $T$. So all translations induced by $\Stab_G(A)$ on the corresponding apartment $A$ have their corresponding translation vectors in the boundary $\ker(\qd)$ of the Tits cone: there is no strongly regular element.

 \par It will be clear from the proof of the theorem that the conclusion of the theorem is true even when $\A^v$ may have an irreducible factor of affine type, if we add the following hypothesis for the action of $G$:

 \medskip\par {\bf(AGT)} For some (and hence for any) apartment $A \in \sha$ of the masure $\Delta$, its stabilizer $\Stab_G(A)$ in $G$ contains an element which induces an \textbf{affinely generic translation}, \ie a translation whose corresponding translation vector $\vect v$ satisfies $\qd(\vect v)\neq 0$ for any imaginary root of an affine component of $\QF\cup\QF_{im}$.
\end{remark}

\begin{proof}[Proof of Theorem~\ref{thm:ExistenceStronglyReg}]
 If $\QD$ is a tree (or an extended tree), then condition (LST) tells us that $G$ is strongly transitive on the associated (quotient) tree, hence contains a strongly regular hyperbolic element. Thus, we assume now that $\QD$ is not an (extended) tree.

We start with a preliminary observation. Let $A$ be an apartment of $\Delta$ and let $H, H'$ be two complementary half-apartments of $A$. We claim that there is some $g \in \Stab_G(A)$ which swaps $H$ and $H'$; in particular $g$ stabilizes the common boundary wall $\partial H = \partial H'$.

In order to prove the claim, choose a pair of opposite panels $\sigma, \sigma'$  at infinity of the wall $\partial H$. Notice that it makes sense to consider panels at infinity since $\Delta$ is not a tree, and thus the twin building $\partial \Delta$ has positive rank.
We now invoke Corollary~\ref{lem:TreeWalls} which provides a canonical isomorphism $\QD(\sigma, \sigma') \cong (T(\sigma, \sigma'), \sha(\sigma, \sigma')) \times \RR^{n-1}$, where $n = \dim(\Delta)$ and $(T(\sigma, \sigma'), \sha(\sigma, \sigma'))$ is a thick tree. The set $\Ch(\sigma)$ is in one-to-one correspondence with $ \partial T(\sigma, \sigma')$.
Recall that $G_{\sigma, \sigma'}$ is strongly transitive on $(T(\sigma, \sigma'), \sha(\sigma, \sigma'))$ by the hypothesis (LST).


Therefore, if $v$ (\resp $D$) denotes the vertex (\resp geodesic line) of $T$ corresponding to the wall $\partial H$ (\resp the apartment $A$), we can find some $g \in \Stab_{G_{\sigma, \sigma'}}(A)$ stabilizing $D$ and acting on it as the symmetry through $v$. It follows that $g$ stabilizes $A$, the wall $\partial H  = \partial H'$ and swaps the two half-apartments $H$ and $H'$. This proves our  claim. (We warn the reader that $g$ might however act non-trivially on the wall $\partial H$.)

\medskip
Using the above proven fact, we claim that, for every root $\alpha \in \QF$, we can construct a translation in $\Stab_G(A)$ of translation vector that is almost collinear to $\alpha^\vee$.

Indeed, let $h$ be a wall of $A$ and denote by $H,H'$ the corresponding two complementary half-apartments of $A$ such that $\partial H = \partial H'=h$. Let $g \in \Stab_{G}(A)$  that swaps $H$ and $H'$ and stabilizes the wall $h$ together with a panel $\sigma$ at infinity of $h$. We denote $r_h \in \Aut(A)$ the reflection with respect to the wall $h$.
Then $r_h\circ g\vert_A$ is a vectorially Weyl--automorphism of $A$ and stabilizes the two chambers of $\partial A$ containing $\sigma$.

 We conclude that $r_h\circ g\vert_A$ is a translation in $\Aut(A)$.
As $g$ and $r_h$ stabilize $h$, the element $r_h\circ g\vert_A$ is a translation parallel to $h$. We call such an element $g  \in \Stab_{G}(A)$ a \textbf{reflection-translation} of wall $h$.

\medskip
We remark the following. Fix two different walls $h_1$ and $h_2$ of the apartment $A$ of direction $\ker\alpha$ and let $g_1, g_2 \in \Stab_{G}(A)$ be two reflection-translations corresponding respectively, to the walls $h_1$ and  $h_2$. We have that $g_1 \circ g_2 \in \Stab_{G}(A)$ is a translation element whose translation vector is not in $\ker\alpha$.
In particular, we notice that the projection 
of the translation vector of $g_1 \circ g_2$ in the direction $\ker\alpha$ equals the sum of the translation vectors  of $g_1$ and $g_2$.
Moreover, the projection of the translation vector of $g_1 \circ g_2$ in the direction of $\alpha^\vee$ depends on the ``euclidean distance'' $\vert\qa(h_2)-\qa(h_1)\vert$ between the walls $h_1$ and $h_2$.

 Let us denote $\gamma: =g_1 \circ g_2 $. For every reflection-translation $g \in \Stab_{G}(A)$ of wall $h_g$ of direction $\ker\alpha$, we have that $\gamma^n g \gamma^{-n} \in \Stab_{G}(A)$ is a reflection-translation of wall $\gamma^n (h_{g})$ and whose translation vector equals the translation vector of $g$. To conclude our last claim, one can choose two walls $h_1$ and $h_2$ of direction $\ker\alpha$ in the same orbit of the group $\langle\qg\rangle$, that are very far away and two reflection-translations $g_1, g_2 \in \Stab_{G}(A)$ corresponding respectively to $h_1$ and $h_2$.
 By the remark above, we have that $g_1 \circ g_2 \in \Stab_{G}(A)$ is a translation and the projection of its translation vector in the direction  $\alpha^\vee$ can be made very big, depending on the distance between $h_1$ and $h_2$.
 In particular, the translation vector of $g_1 \circ g_2$ can be made almost collinear to $\alpha^\vee$ as we want. The claim follows.

Finally, to construct a strongly regular hyperbolic element, we choose a base of roots
 $(\qa_i)_{i\in I}$.
By the hypothesis on the type of $(V,W^v)$, we get some $\Z-$linear combination $\sum_{i\in I}\,n_i\qa_i^\vee$ which is in $C^v_f$, see Kac~\cite[Th.~4.3]{K90}.
But from above, we get a fixed neighborhood $U$ of $0$ in $V$ and elements $g_{i,n}\in \Stab_G(A)$ inducing translations with vector in $nr_i\qa_i^\vee+U$, for every $i \in I$ and some $r_i\in(0,+\infty)$.
Therefore, for some $N_i$ big (with all $N_ir_i/n_i$ almost equal), the product $\prod_{i\in I}\,g_{i,N_i}$ (in any order) will be the desired hyperbolic element.
\end{proof}



\section{Proof of Theorem~\ref{thm::main_theorem_intr}}
\label{sec::main_theorem}



\begin{proposition}(See~\cite[Proposition~2.10]{CaCi})
\label{prop::dynamics_str_reg}
Let $\Delta$ be a masure  and let $\gamma \in \Aut(\Delta)$ be a vectorially Weyl strongly regular hyperbolic element of translation apartment $A$. Let $c_{-} \in \Ch(\bd A_\mp)$ be the unique chamber at infinity that contains the repelling point of $\gamma$ in its interior.

Then for every $c \in \Ch(\bd \Delta_\pm)$ the limit $\lim\limits_{n \to \infty} \gamma^{n}(c)$ exists with respect to the cone topology on $\Ch(\bd \Delta)$ and coincides with the retraction of $c$ onto $A$ centered at the chamber $c_{-}$. In particular, the fixed-point-set of $\gamma$ in $\Ch(\bd \Delta_\pm)$ is the set $\Ch(\bd A_\pm)$.
\end{proposition}
\begin{proof}
The proof goes in the same way as in~\cite[Proposition~2.10]{CaCi}.
\end{proof}

\begin{remark}
\label{rem::semi-discretness}
\par  If $G$ acts strongly transitively on $\QD$, then $G_c^0$ is transitive on $c^{op}$ (apply the definition of the strong transitivity to the situation of (MA2) with $F$ the sector-germ associated with the ideal chamber $c$); in particular Condition (Top) is satisfied.
Therefore, our hypothesis (Top) together with the dynamics at infinity are necessary to obtain Theorem~\ref{thm::main_theorem_intr} { that is recalled below.}
\end{remark}

\begin{theorem*}\textbf{\ref{thm::main_theorem_intr}.}
\label{thm::main_theorem_sec_six}
\textit{Let $(\Delta, \sha)$ be a thick masure and let $G$ be a vectorially Weyl subgroup of $\Aut(\Delta)$.
If $G$ contains a strongly regular hyperbolic element and satisfies Condition (Top), then the following are equivalent:
\begin{enumerate}[(i)]
\item
\label{thm:main-thm-i_int}
$G$ acts strongly transitively on $\Delta$;
\item
\label{thm:main-thm-ii_int}
$G$ acts strongly transitively on the twin building at infinity $\partial \Delta$.
\end{enumerate}}
\end{theorem*}

\begin{proof}(Proof of Theorem~\ref{thm::main_theorem_intr})
The implication $(i)\implies (ii)$ is Proposition \ref{5.5}.
 Let us now prove $(ii)\implies(i)$.
Let $c \in \Ch(\partial \Delta)$, $c' \in c^{op}$ and denote by $A$ the unique apartment in $\Delta$ whose boundary contains $c$ and $c'$.

By hypothesis we have that $G$ contains a strongly regular hyperbolic element. As $G$ is strongly transitive on $\partial \Delta$, $G$ acts transitively on the set $\sha$ of all apartments of $\Delta$, and we conclude that every apartment in $\Delta$ admits a strongly regular hyperbolic element in $G$. Moreover, there exists a strongly regular hyperbolic element $\gamma$ of $ G_{c,c'} \leq G$ such that $c$, respectively $c'$, is the unique chamber at infinity that contains in its interior the repelling, respectively the attracting, point of $\gamma$.

Applying Proposition~\ref{prop::dynamics_str_reg} to the strongly regular element $\gamma\in G_{c,c'}$ with its unique translation apartment $A$ we obtain that $c'$ is an accumulation point for every $G^0_{c} G_{c, c'}$--orbit in $c^{op}$, with respect to the cone topology on $\Ch(\partial \Delta)$.
By our hypothesis, every $G^{0}_c$--orbit in $c^{op}$ is closed in the cone topology on $\Ch(\partial \Delta)$, and so is every $G^0_{c} G_{c, c'}$--orbit in $c^{op}$, by Lemma~\ref{lem:Levi}.
Therefore, there is only one $G^0_{c} G_{c, c'}$--orbit in $c^{op}$. We conclude that $G_{c}=G^0_{c} G_{c, c'}$ and that $G^0_{c} $ is transitive on $c^{op}$. The desired conclusion follows from the criterion of strong transitivity given by Proposition~\ref{lem:ST:criterion}.
\end{proof}

\begin{corollary}
\label{cor::main_theorem}
Let $(\Delta, \sha)$ be a thick masure. Let $G$ be a vectorially Weyl subgroup of $\Aut(\Delta)$. If $\QD$ has some factors of affine type, we ask moreover that $G$ satisfies  the condition (AGT) of Remark~\ref{rem::AGT}. Then the following are equivalent:
\begin{enumerate}[(i)]
\item
\label{cor:main-thm-i}
$G$ acts strongly transitively on $\Delta$;
\item
\label{cor:main-thm-ii}
$G$ acts strongly transitively on the twin building at infinity $\partial \Delta$ {and $G$ satisfies Condition (LST) and Condition (Top) of Definition \ref{1.6}}.
\end{enumerate}
\end{corollary}

\begin{proof}
It follows by applying Theorem~\ref{thm:ExistenceStronglyReg}, Remark~\ref{rem::AGT}, Theorem \ref{thm::main_theorem_intr}, Lemma~\ref{5.7} and Remark~\ref{rem::semi-discretness}.
\end{proof}

Notice that Corollary~\ref{cornoaffinedirectfactors} is a special case of Corollary~\ref{cor::main_theorem}.

\section{Non-validity of the axiom (TTB3) \\ Appendix by  Auguste H{\'e}bert}

Let $(\Delta, \sha)$ be a thick masure and  $X:= \Ch(\bd\QD)$ be its twin building at infinity.
For $\epsilon=+$ or $-$, one denotes by $X_\epsilon$ the set of elements of $\Ch(\bd\QD)$ whose sign is $\epsilon$. Recall that if $c\in \Ch(\bd\QD)$ and $w\in W^v$, one has $E_{\leq w}(c)=\{d\in X_\epsilon|\ \delta(c,d)\leq w\}$, where $\epsilon$ is the sign of $c$.
In Section~\ref{sec::cone_top_hovel}, Ciobotaru, Mühlherr and Rousseau prove that the cone topology satisfies most of the axioms of topological twin buildings of \cite{K02} and \cite{HKM}. We refer to Subsection~\ref{ss:compar} for a summary of the results obtained
in this context.  The aim of the present section is to prove
the following proposition which complements these results:

\begin{proposition}\label{propTTB3 non satisfait}
Let $\Delta$ be a thick masure such that $W^v$ is infinite. We equip $\Ch(\bd\QD)$ with the cone topology. Let $c\in \Ch(\bd\QD)$ and $\epsilon$ be its sign. Then  one has topologically $X_\epsilon\neq \lim\limits_{\rightarrow} E_{\leq w}(c)$. 

In particular, the cone topology on $X_{\epsilon}$ does not
satisfy Axiom (TTB3) of~\cite{HKM} and therefore $\Ch(\partial \Delta)$
is not a topological twin building in the sense of~\cite{HKM}. 
\end{proposition}

In order to prove this, we construct for each chamber $c$ in $\Ch(\bd\QD)$ a set $c \in U\subset\Ch(\bd\QD)$ such that:
\begin{itemize}
\item[-] $U\cap E_{\leq w}(c)$ is open in $E_{\leq w} (c)$ for all $w\in W^v$, with respect to the topology induced by the cone topology on $X_\epsilon$ ,

\item[-] $U$ is not open, with respect to the cone topology on $X_\epsilon$.
\end{itemize}  
\medskip

Let us fix some notation. Let $(\Delta, \sha)$ be a thick masure with $W^v$ infinite. Let $c\in \mathrm{Ch}(\partial \Delta)$ and $\epsilon$ be its sign. We fix an apartment $B \in (\Delta, \sha)$ such that  $c \in \Ch(\partial B)$. We identify B with the model apartment $\A$ of $(\Delta,\mathcal A)$ via a fixed Weyl isomorphism. Then the special point $0\in B$ will be our fixed base point. By Lemma~\ref{lem:3.3} this choice has no influence
 on the definition of the cone topology.  Recall the definition of $U_{x,r,d}$ from Definition~\ref{def::standard_open_neigh_cone_top}.  For $d\in \mathrm{Ch}(\partial \Delta)$ and $r\in [0,\xi_d)$, we just write $U_{r,d}$ instead of $U_{0,r,d}$.

If $w\in W^v$, one sets $E_{\leq w}:=E_{\leq w}(c)$ and  $F_w:=\{ d \in X_{\epsilon} \mid \ell(\delta(c,d)) \leq
\ell(w) \}(=\bigcup_{v\in W^v|\ \ell(v)\leq \ell(w)}E_{\leq v}).$

\medskip
Recall (see Section~\ref{subsubsec::building_infinity}) that if $d\in \Ch(\bd\QD)$ and $x\in \Delta$, then $Q_{x,d}$ is the (unique) sector based at $x$ and whose chamber at infinity is $d$. As we identify $B$ and $\A$, we consider the set of roots $\Phi$ as a set of maps from $B$ to $\R$. Let $\alpha\in \Phi$ be such that $D(\alpha,k)\nsupseteq Q_{0,c}$ for all $k\in \R$, where $D(\alpha,k)=\{x\in B \ |\  \alpha(x)+k\geq 0\}$. Let $k_\alpha\in \Lambda_\alpha\cap \R^*_+$ be such that $k_\alpha\Z  \subset \Lambda_\alpha$. We now fix an $\alpha\in \Phi$ with the above properties.

If $w\in W^v$, one chooses an apartment $A_w$ such that $B{}\cap A_w=D\big(\alpha,k_\alpha\ell(w)\big)$, which is possible by  \cite[Prop~2.9]{Rou11}. By \cite[2.6]{Rou11}, there exists a unique Weyl isomorphism $\phi_w:B{} \rightarrow A_w$ fixing $B{}\cap A_w$ pointwise. One identifies $\R_+$ and $[0,\xi_c)$. Let $r_w\in \R_+$ be such that $D\big(\alpha,k_\alpha\ell(w)\big)\cap [0,\xi_c)=[0,r_w]$. The choice of the apartment $A_w$ is not important for what we do, provided that it satisfies the property above. One can assume moreover that $A_w=A_{w'}$ for all $w,w'\in W^v$ such that $\ell(w)=\ell(w')$, but this is not necessary.

If $A$ is an apartment such that $A\cap B{}$ contains $0$ in its interior, there exists a unique chamber $d$ of $\partial A$ such that $[0,\xi_d)\cap [0,\xi_c) \neq \{0\}$ and we denote it by $c_A$. Indeed, let $d$ be the ideal chamber associated with the enclosure  of the ray of $A$ based at $0$ and containing the germ of $[0,\xi_c)$ in $0$. Then $d$ satisfies $[0,\xi_d)\cap [0,\xi_c)\neq 0$. The uniqueness is a consequence of the fact that two vectorial faces of $\A$ are either equal or disjoint. In particular if $\phi:B{}\rightarrow A$ is an isomorphism fixing a neighborhood of $0$, then $c_A=\phi(c)$.

Let $d\in X_\epsilon$,  $A$ be an apartment of $(\Delta, \sha)$ containing $Q_{0,d}$ and $\phi:B{}\rightarrow A$ be a Weyl isomorphism sending $Q_{0,c}$ on $Q_{0,d}$. Such an isomorphism exists. Indeed, by (MA2), there exists a Weyl isomorphism $\psi: B{}\rightarrow A$ fixing $0$. By definition, there exists some ideal chamber $c'$ of $\partial B{}$ such that $d$ is associated with $\psi(Q_{0,c'})$. By definition, there exists $w\in W^v$
 such that $c'=w.c$. One sets $\phi=\psi\circ w$ and then $\phi(Q_{0,c})=Q_{0,\phi(c)}=Q_{0,\psi(c')}=Q_{0,d}$. If $r\in \R_+$, one sets $[0,r]_d=\phi([0,r])$ and $U_{r,d}=U_{\phi(r),d}$.

\subsection{Construction of a sequence $(d_w)_{w\in W^v}$}

The aim of this subsection is to construct a sequence $(d_w)_{w\in W^v}$ of chambers of sign $\epsilon$ such that for all $w\in W^v$, $d_w\notin F_w$  and $[0,\xi_{d_w})\cap [0,\xi_c)=[0,r_w]$.

An element $s\in W^v$ is called a \textbf{reflection of} $\A$ if it is of the shape $w.r_i.w^{-1}$ for some $w\in W^v$ and $i\in I$. Using Weyl isomorphisms of apartments, we extend this notion to each apartment of $\Delta$. Let $M$ be a wall of $\A$ containing $0$. One writes $M=(w.\alpha_i)^{-1}(\{0\})$ for some $i\in I$ and $w\in W^v$. Then $w.r_i.w^{-1}$ is a reflection fixing $M$ and thus the number of  reflections of $\A$ fixing $0$ is infinite.

\begin{lemma}\label{lemCalcul des distances}
Let $A$ be an apartment such that $A\cap B$ contains $0$ in its interior. Let $A_1$  be an apartment such that $A_1\cap B$ contains $0$ in its interior and such that $A\cap A_1$ is a half-apartment which does not contain $c_A$. Let $M$ be the wall of $A\cap A_1$ and  $A_2=M\cup( A_1\backslash A)\cup (A\backslash A_1)$. Let $s$ be the reflection of $A_2$ fixing $M$. Then

\begin{enumerate}
\item\label{itDescritption des chambres} One has $c_{A_1}=s(c_A)$ in $A_2$.

\item\label{itCalcul de la distance} Let $s':A\rightarrow A$ be a reflection fixing $M$ and $f:A\rightarrow B$ be a Weyl isomorphism sending $c_A$ on $c$. Let $s''\in W^v$ be the vectorial part of  $f\circ s'\circ f^{-1}\in W^v$. Then $\delta(c_A,c_{A_1})=s''$. 
\end{enumerate}
\end{lemma}

\begin{proof}
Let $\phi_0:A_1\rightarrow A_2$ fixing $A_1\cap A_2$ pointwise,  $\phi_1:A \rightarrow A_2$  fixing $A\cap A_2$ pointwise
and $\phi_2:A\rightarrow A_1$ fixing $A\cap A_1$ pointwise. Then by \cite[Lemma~3.4]{H16}, the following diagram is commutative:  \[\xymatrix{ A\ar[d]^{\phi_1}\ar[r]^{\phi_2} & A_1\ar[d]^{\phi_0}\\ A_2\ar[r]^{s}&A_2.}\]
 
 One has $c_{A_1}=\phi_2(c_A)=\phi_0\circ \phi_2 (c_A)$ and $c_{A}=\phi_1(c_A)$, which proves~\ref{itDescritption des chambres}, and~\ref{itCalcul de la distance} is a consequence of~\ref{itDescritption des chambres}.
\end{proof}

\begin{lemma} 
\label{lem::constant_s}
Let $s_\alpha\in W^v$ be the reflection of $B$ with respect to $M(\alpha,0):=\{x\in B|\ \alpha(x)=0\}$. Then for all $w\in W^v$, $\delta(c,c_{A_w})=s_\alpha$.

\end{lemma}

\begin{proof}
Let $w\in W^v$. We apply Lemma~\ref{lemCalcul des distances} with $A=B$ and $A_1=A_w$. Then $f$ is the identity. Let $M_w$ be the wall of $B\cap A_w$ and $s_w$ be the reflection of $B$ fixing $M_w$. Then by Lemma~\ref{lemCalcul des distances}, $\delta(c,c_{A_w})$ is the vectorial part of $s_w$, and the lemma follows.
\end{proof}



\begin{lemma}
Let $w\in W^v$.  Then there exists $d_w\in X_\epsilon \subset \mathrm{Ch}(\partial \Delta)$ such that $d_w\notin F_w$ and $[0,\xi_{d_w}) \cap [0,\xi_c)=[0,r_w]$.
\end{lemma}

\begin{proof}
Let $s''\in W^v$ be a reflection  such that $\ell(s'')>\ell(w)+1+\ell(s_\alpha)$, where $s_\alpha$ is as in Lemma \ref{lem::constant_s}. Let $M''$ be the fixed wall of $s''$ and $M'=\phi_w(M'')$. Let us prove the existence of an apartment $A_1$ satisfying the following conditions: 
\begin{itemize}
\item[-] $A_w\cap A_1$ is a half-apartment containing $[0,r_w]$ in its interior,

\item [-] the wall of $A_w\cap A_1$ is parallel to $M'$,

\item [-] for all $x\in A_w$, $A_w\cap A_1$ does not contain $Q_{x,c_{A_w}}$.
\end{itemize}

Indeed, let $D$ be a half-apartment of $A_w$ satisfying the following conditions: 
\begin{itemize}
\item[-] $D$ contains $[0,r_w]$ in its interior,

\item [-] the wall of $D$ is parallel to $M'$,

\item [-] for all $x\in D$,  $D$ does not contain $Q_{x,c_{A_w}}$.
\end{itemize}
It then suffices to choose an apartment $A_1$ such that $A_1\cap A_w=D$, which is possible by \cite[Prop 2.9]{Rou11}.

Let $d_w=c_{A_1}$. Then $\ell\big(\delta(d_w,c)\big)\geq \ell\big(\delta(d_w,c_{A_w})\big)-\ell\big(\delta(c_{A_w},c)\big)= \ell (s'')-\ell(s_\alpha)\geq \ell(w)+1$. 

By construction, $[0,\xi_{d_w})\cap [0,\xi_c)$ contains $[0,r_w]$. 

Suppose there exists $z\in ([0,\xi_{d_w})\cap [0,\xi_c))\backslash [0,r_w]$. Then $z > r_w$. By \cite[Prop.5.4]{Rou11} $[r_w,z]_{A_1}=[r_w,z]_{B{}}$. Moreover, for $z'\in (r_w,z)_{A_1}$ close enough to $r_w$, we have $z'\in A_w$ (as $[0,r_w]$ is in the interior of $A_w\cap A_1$) and consequently $z'\in A_w\cap{B{}}$. This is absurd because $A_w\cap {B{}}\cap [0,\xi_c)=[0,r_w]$. Therefore $[0,\xi_{d_w}) \cap [0,\xi_c)=[0,r_w]$ and the lemma is proved.
\end{proof}

\subsection{Construction of $U$}

 Let $(d_w)_{w\in W^v}$ be such that for all $w\in W^v$, $d_w\notin F_w$ and $[0,\xi_{d_w})\cap [0,\xi_c)=[0,r_w]$, where $[0,r_w]=[0,\xi_c)\cap D\big(\alpha,k_\alpha\ell(w)\big)$. Let $\shd=\{d_w \ | \ w\in W^v\}$ and $\overline{\shd}=\shd\cup \{c\}$. We now construct a set $U$ containing $c$, such that $U\cap \shd=\emptyset$ and such that $U\cap E_{\leq w}$ is open in $E_{\leq w}$ for all $w\in W^v$, with respect to the topology induced by the cone topology on $X_\epsilon$.

\begin{lemma}\label{lemDescription des ouverts}
Let $d \in \Ch(\partial \Delta)$ and let $ d \in V \subset \Ch(\partial \Delta)$ be an nonempty open subset. Then there exists $r\in \R^*_+$ such that $V\supset U_{r,d}$. 
\end{lemma}

\begin{proof}
By definition, there exist $J\subset \Ch(\bd\QD)$ and $(r_{d'})_{d'\in J}\in (\R^*_+)^{J}$ such that $V=\bigcup_{d'\in J} U_{r_{d'},d'}$. Let $d'\in J$ be such that $d\in U_{r_{d'},d'}$. Then $U_{r_{d'},d}=U_{r_{d'},d'}$, thus $U_{r_{d'},d}\subset  V$, which proves the lemma.
\end{proof}

\begin{lemma}\label{lemIntersection des Urd}
Let $d\in \Ch(\bd\QD)$. Then $\bigcap_{r\in \R^*_+} U_{r,d}=\{d\}$.
\end{lemma}

\begin{proof} By definition, if $c'\in \bigcap_{r\in \R^*_+} U_{r,d}$, then $[0,\xi_{d}) \subset [0,\xi_{c'})$, hence $c'=d$ and the Lemma is proved.
\end{proof}

\begin{lemma}
\label{lem::7.7}
Let $d\in \mathrm{Ch}(\partial \Delta)\backslash \overline{\shd}$. Then there exists $a_d\in \R_+$ such that $U_{a_d,d}\cap \overline{\shd}=\emptyset$. 
\end{lemma}

\begin{proof}
 As $\Ch(\bd\QD)$ with the cone topology is Hausdorff, there exist open sets $V_d\ni d$ and $V_c\ni c$ such that $V_c\cap V_d=\emptyset$. One has $\lim\limits_{\ell(w)\rightarrow +\infty} d_w=c$ and thus for $\ell(w)$ large enough, $d_w\in V_c$. Therefore $V_d\cap \overline{\shd}$ is finite. By Lemma~\ref{lemDescription des ouverts}, one can suppose, reducing $V_d$ if necessary, that $V_d=U_{r,d}$ for some $r\in \R_+$.  Thus Lemma~\ref{lemIntersection des Urd} finishes the proof.
\end{proof}

\begin{lemma}
\label{lem::7.8}
Let $w\in W^v$. Then $(U_{r_w,c}\cap E_{\leq w})\cap \shd$ is empty.
\end{lemma}

\begin{proof}
Let $v\in W^v$. If $\ell(v)<\ell (w)$, then $[0,\xi_{d_v})\cap [0,\xi_c)=[0,r_v]\subsetneq [0,r_w]$ and thus $d_v\notin U_{r_w,c}$. 

If $\ell(w)\leq \ell (v)$, then $d_v\notin F_v$ by construction. As $F_v\supset E_{\leq w}$, the lemma follows.
\end{proof}

Let $w\in W^v$. Using Lemma~\ref{lem::7.8}  one sets  \[U_w:=\bigcup_{d\in (U_{r_w,c}\cap E_{\leq w})\backslash \{c\}} U_{r_w+a_d,d},\] where the  $a_d$'s are as in Lemma~\ref{lem::7.7}.
Let $U:=\bigcup_{w\in W^v} U_w\cup\{c \}$. By construction $U\cap \shd=\emptyset$.

\begin{lemma}\label{lemInclusion de Uw dans Urwc}
Let $w\in W^v$. Then $U_w\subset U_{r_w,c}$.
\end{lemma}

\begin{proof}
Let $d\in U_{r_w,c}$ and $d'\in U_{r_w+a_d,d}$. Then \[ [0,\xi_{d'})\cap [0,\xi_d)\supset [0,r_w+a_d]_{d}\supset [0,r_w]_d=[0,r_w]_c,\] thus $d'\in U_{r_w,c}$ and the lemma follows.
\end{proof}

\begin{lemma}
Let $w\in W^v$. Then $(U_w\cup \{c\})\cap E_{\leq w}=U_{r_w,c}\cap E_{\leq w}$.
\end{lemma}

\begin{proof}
Let $d\in (U_{r_w,c}\cap E_{\leq w})\backslash \{c\}$. Then $d\in U_{r_w+a_d,d}$ and thus \[d\in \bigcup_{d'\in (U_{r_w,c}\cap E_{\leq w})\backslash \{c\}} U_{r_w+a_{d'},d'}=U_w.\] Therefore $U_{r_w,c}\cap E_{\leq w} \subset (U_w\cup \{c\})\cap E_{\leq w}$.

 By Lemma~\ref{lemInclusion de Uw dans Urwc}, $(U_w\cup \{c\})\cap E_{\leq w}\subset U_{r_w,c}\cap E_{\leq w}$ and the lemma follows.
\end{proof}

The following lemma implies Proposition~\ref{propTTB3 non satisfait}.

\begin{lemma}
The set $U$ is not open with respect to the cone topology on $X_\epsilon$, but for all $w\in W^v$, $U\cap E_{\leq w}$ is open in $E_{\leq w}$, with respect to the topology induced by the cone topology on $X_\epsilon$.

\end{lemma}

\begin{proof}
When $\ell(w)\to +\infty$, $r_w\to +\infty$ and thus by the choice of $d_w$, $d_w\to c$. One has  $U\ni c$ but $U\cap \shd =\emptyset$  and thus $U$ is not open with respect to the cone topology on $X_\epsilon$.

Let $w\in W^v$. Then \[U\cap E_{\leq w}=\big( (U_w\cup \{c\}) \cap E_{\leq w}\big)\cup \big(\bigcup_{v\in W^v\backslash \{w\}} (U_v \cap E_{\leq w} )\big)=(U_{r_w,c}\cup \bigcup_{v\in W^v\setminus\{w\}} U_v)\cap E_{\leq w}\] is open in $E_{\leq w}$, with respect to the topology induced by the cone topology on $X_\epsilon$.

\end{proof}

\end{document}